\documentclass[a4paper,10pt,english]{amsart}
\usepackage[left=25mm,top=25mm,bottom=25mm,right=25mm]{geometry}
\usepackage{amsmath,amsfonts,dsfont,amscd,amssymb}
\usepackage{mathrsfs}
\usepackage{lmodern}
\usepackage[utf8]{inputenc} 
\usepackage[english]{babel}
\usepackage{cmap}           
\usepackage[pdfdisplaydoctitle=true,colorlinks=true,urlcolor=blue,citecolor=blue,linkcolor=blue,pdfstartview=FitH,pdfpagemode=UseNone]{hyperref}
\usepackage{hyperxmp}
\usepackage{bm}
\usepackage{graphicx}
\usepackage{enumerate}

\newtheorem{thm}{Theorem}[section]
\newtheorem{cor}[thm]{Corollary}
\newtheorem{lem}[thm]{Lemma}

\newtheorem{prop}[thm]{Proposition}

\theoremstyle{definition}
\newtheorem{defn}[thm]{Definition}
\theoremstyle{remark}
\newtheorem{rem}[thm]{Remark}

\numberwithin{equation}{section}
\newcommand{\RR}{\mathbb{R}}    
\newcommand{\ZZ}{\mathbb{Z}}    

\newcommand{\OO}{\mathrm{O}}    
\newcommand{\SO}{\mathrm{SO}}   
\newcommand{\octa}{\mathbb{O}}  
\newcommand{\DD}{\mathbb{D}}    
\newcommand{\triv}{\mathds{1}}  

\newcommand{\Ela}{\mathbb{E}\mathrm{la}} 
\newcommand{\HH}{\mathbb{H}}        
\newcommand{\Sym}{\mathbb{S}}       
\newcommand{\SV}{\mathsf{V}}        
\newcommand{\SX}{\mathsf{X}}        

\newcommand{\ee}{\bm{e}}            
\newcommand{\nn}{\bm{n}}            
\newcommand{\vv}{\bm{v}}            

\newcommand{\rp}{\mathrm{p}}        

\newcommand{\ba}{\mathbf{a}}
\newcommand{\bb}{\mathbf{b}}
\newcommand{\bc}{\mathbf{c}}
\newcommand{\bd}{\mathbf{d}}
\newcommand{\bk}{\mathbf{k}}
\newcommand{\bq}{\mathbf{q}}
\newcommand{\bt}{\mathbf{t}}
\newcommand{\bv}{\mathbf{v}}
\newcommand{\bz}{\mathbf{0}}

\newcommand{\bA}{\mathbf{A}}
\newcommand{\bB}{\mathbf{B}}
\newcommand{\bC}{\mathbf{C}}        
\newcommand{\bE}{\mathbf{E}}        
\newcommand{\bH}{\mathbf{H}}        
\newcommand{\bS}{\mathbf{S}}        
\newcommand{\bT}{\mathbf{T}}        
\newcommand{\lc}{\bm{\varepsilon}}  

\DeclareMathOperator{\tr}{tr}
\DeclareMathOperator{\sign}{sgn}
\DeclareMathOperator{\Orb}{Orb}

\DeclareMathOperator{\2dots}{:}
\DeclareMathOperator{\3dots}{\raisebox{-0.25ex}{\vdots}}

\newcommand{\rot}{\mathbf{r}}

\newcommand{\strata}[1]{\Sigma_{[#1]}}	                

\newcommand{\norm}[1]{\left\Vert#1\right\Vert}
\newcommand{\abs}[1]{\left\vert#1\right\vert}
\newcommand{\set}[1]{\left\{#1\right\}}

\hypersetup{
  pdfauthor={R. Desmorat, N. Auffray, B. Desmorat, M. Olive, and B. Kolev} 
  pdftitle={Minimal functional bases for elasticity tensor symmetry classes}, 
  pdfsubject={}, 
  pdfkeywords={Anisotropy; Symmetry classes; Higher-order tensors; Polynomial covariants}, 
  pdflang=en, 
  }

\begin{document}

\title{Minimal functional bases for elasticity tensor symmetry classes}%

\author{R. Desmorat}
\address{Université Paris-Saclay, ENS Paris-Saclay, CNRS,  LMT - Laboratoire de Mécanique et Technologie, 91190, Gif-sur-Yvette, France}
\email{rodrigue.desmorat@ens-paris-saclay.fr}
\author{N. Auffray}
\address{MSME, Université Paris-Est, Laboratoire Modélisation et Simulation Multi Echelle, MSME UMR 8208 CNRS, 5 bd Descartes, 77454 Marne-la-Vallée, France}
\email{Nicolas.auffray@univ-mlv.fr}

\author{B. Desmorat}
\address{Sorbonne Université, UMPC Univ Paris 06, CNRS, UMR 7190, Institut d'Alembert, F-75252 Paris Cedex 05, France \& Univ Paris Sud 11, F-91405 Orsay, France}
\email{boris.desmorat@upmc.fr}

\author{M. Olive}
\address{Université Paris-Saclay, ENS Paris-Saclay, CNRS,  LMT - Laboratoire de Mécanique et Technologie, 91190, Gif-sur-Yvette, France}
\email{marc.olive@math.cnrs.fr}

\author{B. Kolev}
\address{Université Paris-Saclay, ENS Paris-Saclay, CNRS,  LMT - Laboratoire de Mécanique et Technologie, 91190, Gif-sur-Yvette, France}
\email{boris.kolev@math.cnrs.fr}

\date{January 19, 2022}%

\thanks{Three of the authors, R. Desmorat, B. Kolev and M. Olive, were partially supported by CNRS Projet 80--Prime GAMM (Géométrie algébrique complexe/réelle et mécanique des matériaux).}


\begin{abstract}
  Functional bases, synonymous with separating sets, are usually formulated for an entire vector space, such as the space $\Ela$ of elasticity tensors. We propose here to define functional bases limited to symmetry strata, \emph{i.e.}, sets of tensors of the same symmetry class. We provide such low-cardinality minimal bases for tetragonal, trigonal, cubic or transversely isotropic symmetry strata of the elasticity tensor.
\end{abstract}

\maketitle

\section{Introduction}

In the field of linear elasticity, the mechanical properties of an elastic material are represented by an elasticity tensor $\bE$, element of the vector space $\Ela$. This association is nevertheless not unique since two elasticity tensors, that differ only up to a rotation, describe the same elastic material \cite{FV1996}. It is important, for applications, to be able to distinguish within $\Ela$ which tensors represent the same materials from those who do not. The answer to this question is provided by the construction of a finite set $\mathscr{F}$ -- preferably minimal -- of $\SO(3)$-invariant functions (simply called invariant functions in the following), which
\begin{enumerate}
  \item enable one to check if two elasticity tensors describe the same elastic material, \emph{i.e.},  that they are related by a rotation;
  \item allow one to rewrite any invariant function $f$ of an elasticity tensor $\bE$ as a function of the elements of $\mathscr{F}$
        (\emph{i.e.}, rewrite $f(\bE)=F(\mathscr F)$ for some function $F$).
\end{enumerate}
This second point constitutes the core of the application of Invariant Theory to Continuum Mechanics~\cite{Smi1971,Ver1979,Boe1987,Via1997,VV2001}.

The knowledge of an \emph{integrity basis} provides an answer to this twofold question, but, generally, the cardinality of a minimal integrity basis can be very high. For instance, in the case of three-dimensional elasticity, a minimal integrity basis consists of 294 elements~\cite{OKA2017,OKDD2021}. This is mainly due to the fact that an integrity basis is a response to a different mathematical question, namely, \emph{the determination of a set of generators for the algebra of $\SO(3)$-invariant polynomial functions over $\Ela$}\footnote{Any invariant polynomial in the components $E_{ijkl}$ of $\bE$ can be written as a polynomial in the elements of the integrity basis of the elasticity tensors.}.

An invariant set which satisfies (1) is called a \emph{separating set}, while one which satisfies (2) is called a \emph{functional basis}~\cite{Wey1997}. Although they seem different at first glance, these two notions are in fact equivalent, as shown by  Wineman and Pipkin~\cite{WP1964}. This is interesting since the cardinality of a functional basis can be lower than the one of an integrity basis. But, in contrast to integrity bases and despite some attempts~\cite{DKW2008,Olive2017}, there is no general algorithm to obtain functional bases.

For isotropic elasticity, it is well-known that Lamé parameters $\lambda, \mu$ are two invariants that allow to separate isotropic elasticity tensors and to write invariant functions of an isotropic elasticity tensor $\bE$ (any invariant function $f(\bE)$ can be written as $f(\bE)=F(\lambda,\mu)$ for some function $F$). The extension of this simple observation to the whole vector space $\Ela$ is a difficult problem, as emphasized by Ming \emph{et al}~\cite{MCQ+19}. Indeed, these authors have obtained a polynomial functional basis of 251 elements, still a rather large number!
There are in the literature different strategies to reduce the number of elements of a functional basis. For instance,
\begin{itemize}
  \item \emph{change the class of its elements}: usually polynomial invariants are considered~\cite{Smi1971,Zhe1994,OA2014,OKDD2021,CLQZZ2018,LDQZ2018,MZC2019}, but this is not mandatory;
  \item \emph{look for local separating sets instead of global ones}: the separating property is then defined, not on the whole vector space, but only on a neighbourhood of a given tensor. In this direction, Bona et al.~\cite{BBS2008} proposed a local parametrization of orbits of generic triclinic elasticity tensors by 18 \emph{local algebraic invariants}. A separating set of 18 \emph{local polynomial invariants} was provided in~\cite[Theorem A.3]{DADKO2019};
  \item \emph{restrict the separating property to a subset of generic tensors} (generally triclinic). The corresponding functional bases are then called \emph{weak functional bases}~\cite{BKO1994}.
\end{itemize}
When combined, these strategies lead to a drastic reduction in the cardinality of a functional basis. For three-dimensional elasticity tensors, \emph{a weak separating set of 39 global polynomial invariants} has been provided in~\cite{BKO1994}, and \emph{a weak separating set of 18 global rational invariants} has been obtained in~\cite[Corollary 4.5]{DADKO2019}. Nevertheless, to reduce this set from 294 elements to only 18, a price has to be paid, some (in general non triclinic) elasticity tensors are \textit{a priori} excluded from the possibility to check them.

The approach followed here is complementary. Instead of considering the whole vector space $\Ela$, we are looking for sets of invariants which separate tensors of a given symmetry class, with no genericity restrictions. Our aim is then to produce optimal functional bases, on these lower-dimensional elasticity symmetry classes of $\Ela$. In this paper, we will achieve this task for trigonal, tetragonal, transverse isotropic, and cubic elasticity tensors. Our work strongly relies on the geometry of fourth-order harmonic tensors~\cite{AKP2014} and elasticity tensors~\cite{OKDD2021}.

\subsection*{Outline}

The eight symmetry classes of linear elasticity and the associated breaking symmetry diagram (due to~\cite{FV1996}) are recalled in~\autoref{sec:cov-charac-Ela}, where we summarize necessary and sufficient polynomial conditions (obtained in~\cite{OKDD2021}) for an elasticity tensor to belong to a given \emph{symmetry stratum} (\emph{i.e.}, a set of elasticity tensors of the same symmetry class).
In~\autoref{sec:FBTS}, we introduce the mathematical material necessary to define rigorously the notion of \emph{minimal functional bases}, not only on the whole elasticity tensor space $\Ela$ but also -- and this is the originality of the present work -- on each symmetry stratum. We illustrate this method, first in~\autoref{sec:S2}, by the construction of minimal functional bases for the orthotropic and the transversely isotropic strata of the space of \emph{second-order symmetric tensors}, and, then, in~\autoref{sec:H4}, by one for the orthotropic, the tetragonal, the trigonal and the transversely isotropic strata of the space of \emph{fourth-order harmonic tensors} (which appear in the harmonic decomposition of elasticity tensors). Thanks to the key-definition of a non vanishing second-order covariant, we obtain, in an intrinsic manner, our main result in~\autoref{sec:Ela} and~\autoref{sec:AllFivesClasses}, which is the explicit formulation of low-cardinality functional bases for elasticity tensors at least tetragonal or trigonal.

\subsection*{Tensorial operations}
\label{subsec:notations}

Using the Euclidean structure of $\RR^{3}$, no distinction will be made between covariant, contravariant or mixed tensors. All tensor components will be expressed with respect to an orthonormal basis $(\ee_{i})$. The space of $n$th-order tensors will be denoted by $\otimes^{n}(\RR^{3})$, and the subspace of totally symmetric tensors of order $n$ by $\Sym^{n}(\RR^{3})$. A traceless tensor $\bH \in \Sym^{n}(\RR^{3})$ is called an \emph{harmonic tensor} and the space of $n$th-order harmonic tensors is denoted by $\HH^{n}(\RR^{3})$.

The \emph{contraction} over two or three indices between second/fourth-order tensors will be denoted by
\begin{equation*}
  \begin{array}{ll}
    \ba\2dots\bb = a_{ij}b_{ij},                  & (\bA\2dots\ba)_{ij} = A_{ijkl}a_{kl} ,    \\
    (\bA \2dots \bB)_{ijkl} = A_{ijpq}B_{pqkl}  , & (\bA \3dots \bB)_{ij} = A_{ipqr}B_{pqrj}.
  \end{array}
\end{equation*}
The \emph{total symmetrization} of an $n$th-order tensor $\bT$ is the tensor $\bT^{s}$, defined by
\begin{equation*}
  (\bT^{s})_{i_{1}\dotsc i_{n}} =  \frac{1}{n!}\sum_{\sigma \in \mathfrak{S}_{n}} T_{i_{\sigma(1)} \dotsc i_{\sigma(n)}}
  \in \Sym^{n}(\RR^{3}),
\end{equation*}
where $\mathfrak{S}_{n}$ is the permutation group over $n$ elements.

The \emph{symmetric tensor product}, noted $\odot$, and the \emph{generalized cross product} (introduced in~\cite{OKDD2018}), noted $\times$, between two totally symmetric tensors $\bS_{1}\in \Sym^{n_{1}}(\RR^{3})$ and $\bS_{2}\in \Sym^{n_{2}}(\RR^{3})$, are defined respectively by
\begin{align}\label{eq:symmetric-tensor-product}
   & \bS_{1} \odot \bS_{2} := ( \bS_{1} \otimes \bS_{2})^{s}
  \; \in \Sym^{n_{1}+n_{2}}(\RR^{3}),
  \\
  \label{eq:cross-product}
   & \bS_{1} \times \bS_{2} := (\bS_{2} \cdot \lc \cdot \bS_{1})^{s}\in \Sym^{n_{1}+n_{2}-1}(\RR^{3}),
\end{align}
where $\lc$ is the third-order Levi-Civita tensor (with components $\varepsilon_{ijk}=\det(\ee_{i}, \ee_j, \ee_{k})$). Explicit component formulas for the generalized cross product involving second and fourth-order tensors can be found in~\cite{ADDKO2019}. We have moreover~\cite{OKDD2018}
\begin{equation}\label{eq:Sxq}
  \bS \times \bq=0, \quad \forall \bS\in \Sym^{n}(\RR^{3}),
\end{equation}
where $\bq=(\delta_{ij})$ is the Euclidean metric.

\begin{figure}
  \centering
  \includegraphics[scale=1]{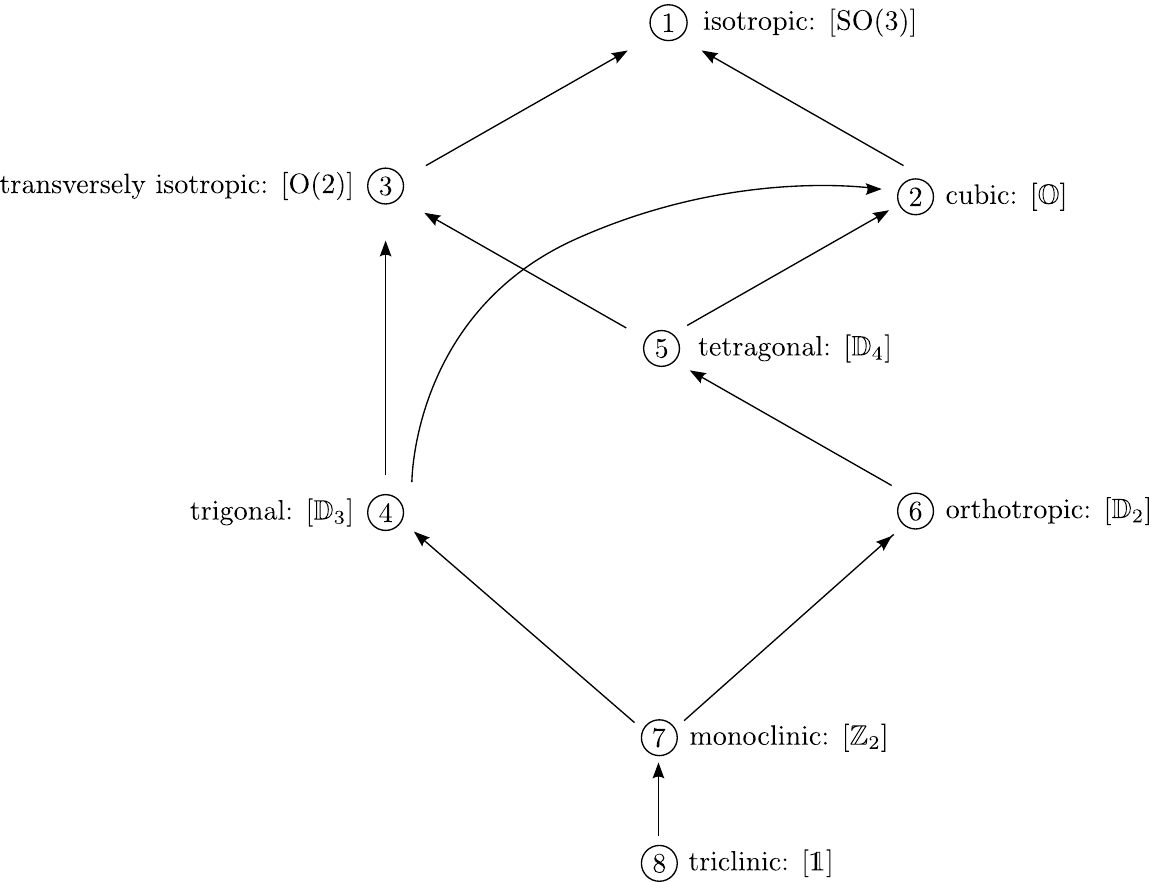}
  \caption{Symmetry classes of elasticity tensors and of fourth-order harmonic tensors~\cite{IG1984,FV1996} (figure from~\cite{AKP2014}).}
  \label{fig:bifurcationEla}
\end{figure}

\section{Covariant characterization of elasticity symmetry classes}
\label{sec:cov-charac-Ela}

Let
\begin{equation*}
  \Ela := \set{\bE\in\otimes^{4}(\RR^{3});\; E_{ijkl}=E_{klij}=E_{jikl}}
\end{equation*}
be the 21-dimensional vector space of three-dimensional elasticity tensors. It is endowed with the natural $\SO(3)$ representation given by
\begin{equation}\label{eq:gE}
  (g\star \bE)_{ijkl} := g_{ip}g_{jq}g_{kr}g_{ls} E_{ijkl}, \quad  g\in \SO(3).
\end{equation}

\subsection{Elasticity symmetry classes and strata}

Forte and Vianello~\cite{FV1996} have shown that there are exactly eight different elasticity symmetry classes, depicted in~\autoref{fig:bifurcationEla}, and in which the mechanical names are provided aside the associate group designation $[H]$:  triclinic $[\triv]$, monoclinic $[\ZZ_{2}]$, orthotropic $[\DD_{2}]$, tetragonal $[\DD_{4}]$, trigonal $[\DD_{3}]$, transversely-isotropic $[\OO(2)]$, cubic $[\octa]$ and isotropic $[\SO(3)]$ (see~\autoref{sec:Ela-symmetry-groups} for the group notations).

Given a symmetry class $[H]$, the symmetry stratum $\strata{H}$ is the set of all the elasticity tensors which have exactly the symmetry class $[H]$. Observe, for instance, that a transversely isotropic elasticity tensor $\bE$ has also tetragonal symmetry. In such a case, we will say that $\bE$ is \emph{at least tetragonal}, but it does not belong to the tetragonal stratum $\strata{\DD_{4}}$. This ``at least'' order relation is depicted by the arrows of~\autoref{fig:bifurcationEla}.

\subsection{Harmonic decomposition -- Covariants}

The first step, when studying the geometry of elasticity tensors, consists in splitting $\Ela$ into stable, irreducible vector spaces (under the action of $\SO(3)$). This is the so-called \emph{harmonic decomposition}~\cite{Bac1970}. Introducing the second-order \emph{dilatation tensor}
\begin{equation*}
  \bd := \tr_{12}\bE, \qquad d_{ij} = E_{kkij},
\end{equation*}
and the second-order \emph{Voigt tensor}
\begin{equation*}
  \bv := \tr_{13}\bE, \qquad v_{ij}=E_{kikj}
\end{equation*}
one obtains an explicit harmonic decomposition of $\bE$ (see~\cite{Cow1989a,Cow1989a,Bae1993,FV1996,ADDKO2019}),
\begin{equation}\label{eq:dec-harm-E}
  \bE = (\tr \bd, \tr \bv, \bd^{\prime}, \bv^{\prime}, \bH).
\end{equation}
In this decomposition, the harmonic components are the two scalar invariants
\begin{equation}\label{eq:dec-harm-lambdamu}
  \tr \bd,  \qquad  \tr \bv,
\end{equation}
the two deviatoric tensors
\begin{equation}\label{eq:dec-harm-dv}
  \bd^{\prime} = \bd - \frac{1}{3} (\tr \bd) \, \bq, \qquad \bv^{\prime} = \bv - \frac{1}{3} (\tr \bv) \, \bq,
\end{equation}
and the harmonic (\emph{i.e.}, totally symmetric and traceless) fourth-order tensor
\begin{equation}\label{eq:H4E}
  \bH = \bE^{s} - \bq \odot \ba^{\prime} - \frac{7}{30} (\tr \ba)\, \bq \odot \bq, \quad \ba := \frac{2}{7}(\bd+2\bv),
\end{equation}
where $\bE^s$ is the totally symmetric part of $\bE$, and $\odot$ is the symmetrized tensor product defined in~\eqref{eq:symmetric-tensor-product}. The harmonic decomposition \eqref{eq:dec-harm-E} is equivariant, meaning that it satisfies:
\begin{equation*}
  g \star \bE = (g\star \tr \bd , g \star \tr \bv, g  \star \bd^{\prime}, g \star \bv^{\prime}, g \star \bH)=(\tr \bd, \tr \bv, g  \star \bd^{\prime}, g \star \bv^{\prime}, g \star \bH),
\end{equation*}
for any rotation $g\in \SO(3)$. Note here that $g\star \lambda =\lambda$ for scalar invariants $\lambda$. The action of a rotation on a second-order tensor $\ba$ is $g\star \ba =g \ba g^{t}$, while the action of a rotation on a fourth-order tensor is given by~\eqref{eq:gE}. The harmonic components
\begin{equation*}
  \tr \bd=\tr (\bd(\bE)), \quad \tr \bv=\tr (\bv(\bE)), \quad \bd^{\prime}=\bd^{\prime}(\bE), \quad \bv^{\prime}= \bv^{\prime}(\bE), \quad \bH=\bH(\bE),
\end{equation*}
are \emph{covariants $\bC(\bE)$} of $\bE$~\cite{KP2000,OKDD2021} (of respective order 0, 0, 2, 2 and 4, $\tr \bd$ and $\tr \bv$ being scalar invariants of $\bE$, and $\bd^{\prime}(\bE)$, $\bv^{\prime}(\bE)$ and $\bH=\bH(\bE)$ being linear covariants of $\bE$). They satisfy the rule
\begin{equation*}
  \bC(g\star \bE) = g\star \bC(\bE), \qquad \forall g\in \SO(3).
\end{equation*}
However, there also exists \emph{polynomial covariants} of higher degree. For instance, the quadratic covariant
\begin{equation}\label{eq:d2}
  \bd_{2}(\bH):=\bH\3dots \bH , \qquad (\emph{i.e.}, \; (\bd_{2})_{ij}=H_{ipqr}H_{pqrj}),
\end{equation}
introduced by Boehler, Kirillov and Onat in 1994~\cite{BKO1994}, and which plays a fundamental role in the classification (by symmetry classes) of the fourth-order harmonic tensor and of the elasticity tensor. Indeed, necessary and sufficient conditions for an elasticity tensor to be of a given symmetry class have been formulated in~\cite{OKDD2021}, involving $\bd$, $\bv$, $\bd_{2}$ and other higher degree polynomial covariants.

\subsection{Covariant characterization of elasticity symmetry classes}

The following theorem was proved in~\cite[Theorem 10.2]{OKDD2021}. It provides a characterization of the isotropic, cubic, transversely isotropic, tetragonal and trigonal symmetry classes of elasticity (that is for elasticity tensors which are at least trigonal or tetragonal). We denote by $\ba'=\ba-\frac{1}{3} (\tr \ba) \; \bq$, the deviatoric part of a symmetric second-order tensor $\ba$ and recall that $\bH \times \bq=0$, so that $\bH \times \ba=\bH \times \ba'$.

\begin{thm}\label{thm:main}
  Let $\bE =(\tr \bd, \tr \bv, \bd^{\prime}, \bv^{\prime}, \bH) \in \Ela$ be an elasticity tensor. Then
  \begin{enumerate}
    \item $\bE$ is isotropic if and only if $\bd^{\prime} = \bv^{\prime} = \bd_{2} = 0$.

    \item $\bE$ is cubic if and only if $\bd^{\prime} = \bv^{\prime} = \bd_{2}^{\prime} = 0$ and $\bd_{2} \ne 0$.

    \item $\bE$ is transversely isotropic if and only if $(\bd_{2}, \bd, \bv)$ is transversely isotropic and
          \begin{equation*}
            \bH \times \bd_{2} = \bH \times \bd = \bH \times \bv = 0.
          \end{equation*}

    \item $\bE$ is tetragonal if and only if $(\bd_{2}, \bd, \bv)$ is transversely isotropic,
          \begin{equation*}
            \tr(\bH \times \bd_{2}) = \tr(\bH \times \bd) = \tr(\bH \times \bv) = 0,
          \end{equation*}
          and
          \begin{equation*}
            \bH \times \bd_{2} \ne 0, \quad \text{or} \quad \bH \times \bd \ne 0, \quad \text{or} \quad \bH \times \bv \ne 0.
          \end{equation*}

    \item $\bE$ is trigonal if and only if $(\bd_{2}, \bd, \bv)$ is transversely isotropic,
          \begin{equation*}
            \bd_{2} \times (\bH\2dots\bd_{2}) = \bd \times (\bH\2dots\bd) = \bv \times (\bH\2dots\bv) = 0,
          \end{equation*}
          and
          \begin{equation*}
            \tr( \bH \times \bd_{2}) \ne 0, \quad \text{or} \quad \tr(\bH \times \bd) \ne 0, \quad \text{or} \quad \tr(\bH \times \bv) \ne 0.
          \end{equation*}

  \end{enumerate}
\end{thm}

As a corollary of this theorem, we have the following result.

\begin{cor}\label{cor:transversely-isotropic}
  Let $\bE$ be an elasticity tensor which is either transversely isotropic, tetragonal or trigonal. Then,
  $(\bd,\bv,\bd_{2})$ is transversely isotropic (or equivalently $(\bd^{\prime},\bv^{\prime},\bd_{2}^{\prime})$ is transversely isotropic). In particular, there exists a unit vector $\nn$, defining the axis $\langle \nn \rangle$ of transverse isotropy, and such that
  \begin{equation*}
    \bd^{\prime}=\alpha (\nn\otimes \nn)',\qquad \bv^{\prime}=\beta (\nn\otimes \nn)',\qquad
    \bd_{2}^{\prime}=\gamma (\nn\otimes \nn)',
  \end{equation*}
  where $(\alpha,\beta, \gamma)\ne (0,0,0)$.
\end{cor}

\section{Functional bases and separating sets}
\label{sec:FBTS}

In this section, we recall basic notions in Invariant Theory, in particular: \emph{functional basis}, \emph{separating set} and \emph{integrity basis}, and the associated notion of \emph{minimality}. The concepts of functional basis and separating set are meaningful in a very general setting, namely for the action of a group $G$ on a set $\SX$~\cite{Wey1997}, and are moreover equivalent, as noted by Wineman and Pipkin~\cite{WP1964}. Defining a finite integrity basis requires some additional structure, for instance that $G$ is a compact Lie group~\cite{Bredon1993} (with the remark that in solid mechanics, many relevant groups are  compact), $\SX= \SV$ is a vector space, and the action of $G$ on $\SV$ is linear.

\subsection{Action of a group on a set}
\label{subsec:Action_group}

An action $\star$ of a group $G$ on a set $\SX$ is a mapping
\begin{equation*}
  G \times \SX \to  \SX,
  \qquad
  (g, x) \mapsto g \star x,
\end{equation*}
such that
\begin{equation*}
  (g_{1} g_{2}) \star x = g_{1} \star (g_{2} \star x),
  \qquad
  e \star x =x,
\end{equation*}
where $g_{1}, g_{2} \in G$ and $e$ is the unit element of $G$. When $\SX=\SV$ is a vector space and the action is linear in $x$, such an action is called a \emph{linear representation} of $G$ on $\SX$. The \emph{symmetry group} of $x$ (also known as the isotropy group of $x$)  is defined as $G_{x}:=\set{g\in G,\, g\star x=x}$
and the \emph{symmetry class} of $x$, noted $[G_x]$, is defined as the conjugacy class of $G_x$ in $G$, \textit{i.e.}
\begin{equation*}
  [G_x]:=\set{gG_{x}g^{-1},\, g\in G}.
\end{equation*}
A \emph{symmetry stratum} $\Sigma_{[H]}$ is the set of all elements $x$ with symmetry group $G_{x}$ conjugate to $H$:
\begin{equation*}
  \Sigma_{[H]}:=\set{x\in X,\quad G_x\in [H]}.
\end{equation*}
The orbit of the point $x\in \SX$ is defined as the set
\begin{equation*}
  \Orb(x):=\set{g\star x,\quad g\in G}.
\end{equation*}
Observe that all points in $\Orb(x)$ belong to the same symmetry stratum, since
$G_{g\star x}=gG_{x}g^{-1}$.
Finally, the \emph{orbit space} $\SX/G$ is the set of orbits and the canonical projection is the mapping
\begin{equation}\label{eq:Quotient_Map}
  \pi \: : \: \SX \longrightarrow \SX/G,
  \qquad
  x \mapsto \Orb(x).
\end{equation}

\subsection{Functional bases and separating sets}

The action of $G$ on $\SX$ induces a linear action of $G$ on the \emph{vector space} $\mathcal{F}(\SX)$ of real-valued functions on $\SX$, which is written
\begin{equation*}
  (g\star f)(x):=f(g^{-1}\star x),
\end{equation*}
where $f\in \mathcal{F}(\SX)$ and $g\in G$. The algebra $\mathcal{F}(\SX)^{G}$ of $G$-invariant functions on $\SX$ is defined by
\begin{equation}\label{eq:FXG}
  \mathcal{F}(\SX)^{G}:=\set{f\in \mathcal{F}(\SX),\quad g\star f=f,\quad \forall g\in G},
\end{equation}
and this definition leads to the notion of functional basis for $G$-invariant functions on $\SX$. This notion, introduced in Weyl's classical book~\cite{Wey1997}, has become a key notion in the mechanical science literature related to Invariant Theory~\cite{WP1964,Smi1971,Boe1987,Zhe1994}.

\begin{defn}[Functional basis]\label{def:Func_Basis}
  A finite set $\mathscr{F}:=\{\varphi_{1},\dotsc,\varphi_{s}\}$ of $G$-invariant functions is a functional basis of $\mathcal{F}(\SX)^{G}$ if for any $G$-invariant function $f\in \mathcal{F}(\SX)^{G}$ there exists a function $F:\RR^{s} \to \RR$ such that
  \begin{equation*}
    f(x)=F(\varphi_{1}(x),\dotsc,\varphi_{s}(x)),\quad \forall x\in \SX.
  \end{equation*}
  A functional basis $\mathscr{F}$ is said to be \emph{minimal} if no proper subset $\mathscr{F}'$ of $\mathscr{F}$ is a functional basis.
\end{defn}

As pointed out by Weyl~\cite[Page 30]{Wey1997}, the word function has to be understood in its widest scope. Such a function $F$ may not even be continuous~\cite[Section 5]{Smi1971}.

\begin{defn}[Separating set]
  A finite set $\mathscr{S}:=\set{\kappa_{1},\dotsc,\kappa_{r}}$ of $G$-invariant functions is a \emph{separating set} of $\SX/G$ if for any $x,\overline{x}$ in $\SX$
  \begin{equation*}
    \Orb(x)=\Orb(\overline{x}) \iff \kappa_{i}(x)=\kappa_{i}(\overline{x}),\quad i=1,\dotsc,r.
  \end{equation*}
  A separating set $\mathscr{S}$ is said to be \emph{minimal} if no proper subset $\mathscr{S}'$ of $\mathscr{S}$ is a separating set.
\end{defn}

Given a separating set $\set{\kappa_{1},\dotsc,\kappa_{r}}$ of invariant functions, the mapping
\begin{equation}\label{eq:Map_Eval_Inv}
  K\: : \: \SX\longrightarrow \RR^{r},\quad x\mapsto (\kappa_{1}(x),\dotsc,\kappa_{r}(x)).
\end{equation}
induces an injective mapping from the orbit space $\SX/G$ into $\RR^{r}$ and one has the following result~\cite{WP1964} (see also~\cite{PR1959,PW1963}).

\begin{thm}[Wineman and Pipkin]\label{thm:sepimpliesfunc}
  Consider a group $G$ acting on a set $\SX$. Then, each separating set $\set{\kappa_{1},\dotsc,\kappa_{r}}$ of $\SX/G$ is a functional basis of $\mathcal{F}(\SX)^{G}$: for each $G$-invariant function $f$, there exists a function
  \begin{equation*}
    F\: : \: \text{\emph{Im}}(K)\longrightarrow \RR,\quad \text{\emph{Im}}(K):=\set{K(x);\; x\in \SX},
  \end{equation*}
  such that
  \begin{equation*}
    f(x)=F(\kappa_{1}(x),\cdots,\kappa_{r}(x)),\quad \forall x\in \SX.
  \end{equation*}
  Conversely, each functional basis $\mathscr{F}=\set{\varphi_{1},\dots,\varphi_{s}}$ of $\mathcal{F}(\SX)^{G}$ is also a separating set of $\SX/G$.
\end{thm}

Note that the cardinality of a minimal separating set/functional basis is not well-defined. It may vary from one minimal set to another. Besides, a lower bound on the cardinality of such a set depends drastically on the class of functions (continuous, differentiable, \ldots) for which it is defined. For instance, Wang~\cite{Wan1970a} (see also~\cite[p.39]{Boe1987}) has noticed that, by omitting continuity, it is always possible to construct a separating set \emph{of only one element}. On the other side, if $\SX/G$ is (at least) a topological manifold and the class of invariant functions considered are at least continuous, then the cardinality of a functional basis is at least the dimension of the quotient space $\SX/G$, as detailed in the following remark.

\begin{rem}\label{rem:Minimality_and_Dimension}
  When the orbit space $\SX/G$ is a \emph{topological manifold} of dimension $d$, the cardinality of any separating set $\set{\kappa_{1},\dots,\kappa_{r}}$ of continuous functions is bigger than the dimension of $\SX/G$ ($r\geq d$). This is a consequence of the \emph{invariance of domain theorem}~\cite{Brouwer1912,HY1988}, which states that if there is a \emph{continuous injective} mapping $f$ from an open subset $U$ of $\RR^{d}$ into $\RR^{r}$, then, necessarily $r\geq d$.
\end{rem}

\subsection{Linear representation of a compact Lie group}

From now on, we focus on a linear action of a \emph{compact Lie group} $G$ on a vector space $\SV$ (usually called a \emph{linear representation of $G$ on $\SV$}). In that case, there exists only a finite number of symmetry classes $[H_{1}],\dots,[H_l]$ and $\SV$ splits into a disjoint union of strata~\cite{AS1983,Bredon1960}
\begin{equation*}
  \SV=\Sigma_{[H_{1}]}\cup \dotsc \cup\Sigma_{[H_l]},
\end{equation*}
where each stratum $\Sigma_{[H]}$ is a $G$-stable smooth submanifold of $\SV$~\cite{Bre1972,AS1983,Olver1995,AKP2014}.

We shall denote by $\RR[\SV]$, the \emph{algebra} of polynomial functions on $\SV$, and by
\begin{equation*}
  \RR[\SV]^{G}:=\set{\rp\in \RR[\SV];\; \rp(g\star \vv)=\rp(\vv),\quad \forall g\in G,\, \forall\vv\in \SV},
\end{equation*}
the subalgebra of $\RR[\SV]$ consisting of polynomial invariants. As a consequence of Hilbert's finiteness theorem~\cite{Hil1993,Stu2008}, the algebra $\RR[\SV]^{G}$ is finitely generated and any finite set $\set{I_{1}, \dotsc , I_{N}}$ of generators is called an \emph{integrity basis}. We recall that the generating property means that each $G$-invariant polynomial $J\in\RR[\SV]^{G}$ is a polynomial function in $I_{1},\dotsc,I_{N}$:
\begin{equation*}
  J(\vv) = \rp(I_{1}(\vv), \dotsc ,I_{N}(\vv)), \qquad \vv \in \SV,
\end{equation*}
where $\rp$ is a polynomial in $N$ variables. An integrity basis is \emph{minimal} if no proper subset of it is an integrity basis.

As we are dealing with linear representations of a compact Lie group on a real vector space, any integrity basis is also a separating set of the orbit space $\SV/G$ (see~\cite[Appendix C]{AS1983}), and is thus a functional basis of $\mathcal{F}(\SV)^{G}$.

We will end this section by formulating a theorem which will be helpful to achieve our goal which is to produce minimal functional bases for the stable subsets $\Sigma_{[H]}$ of $\SV$, rather than for $\SV$ itself.

\begin{thm}\label{thm:Key_Thm_Separating_Set}
  Let $\mathcal{B}:=\set{I_{1}, \dotsc , I_{N}}$ be an integrity basis of $\RR[\SV]^{G}$, and $\Sigma_{[H]}$, a symmetry stratum with $d=\dim(\Sigma_{[H]}/G)$. Suppose that there exist $G$-invariant continuous functions $\kappa_{1},\dots,\kappa_d$ in $\mathcal{F}(\Sigma_{[H]})^{G}$ and functions $F_{1},\dots,F_{N}$ such that
  \begin{equation*}
    I_{k}(\vv)=F_{k}(\kappa_{1}(\vv),\dotsc,\kappa_{d}(\vv)),\quad \forall \vv\in \Sigma_{[H]},\quad \forall k=1,\dotsc,N.
  \end{equation*}
  Then $\set{\kappa_{1},\dotsc,\kappa_{d}}$ is a minimal separating set of $\Sigma_{[H]}/G$ and a minimal functional basis of $\mathcal{F}(\Sigma_{[H]})^{G}$.
\end{thm}

\begin{proof}
  As already noticed, for a real representation of a compact Lie group, an integrity basis $\mathcal{B}$ is also a separating set of $\SV/G$~\cite[Appendix C]{AS1983}. By hypothesis, for any $\vv,\overline{\vv}\in \Sigma_{[H]}$
  \begin{equation*}
    \forall i,\quad \kappa_{i}(\vv)=\kappa_{i}(\overline{\vv}) \implies \forall k,\quad I_{k}(\vv)=I_{k}(\overline{\vv}).
  \end{equation*}
  Hence, $\Orb(\vv)=\Orb(\overline{\vv})$, and we deduce that the set $\set{\kappa_{1},\dotsc,\kappa_{d}}$ is a separating set of $\Sigma_{[H]}/G$, as well as a functional basis of $\mathcal{F}(\Sigma_{[H]})^{G}$ by theorem~\ref{thm:sepimpliesfunc}. Finally, the minimality is a direct consequence of remark~\ref{rem:Minimality_and_Dimension}.
\end{proof}

\section{Functional bases on symmetry strata of second-order tensors}
\label{sec:S2}

Let us first illustrate the notions introduced in~\autoref{sec:FBTS} for the standard action of the rotation group $G = \SO(3)$ on the vector space $\SV=\Sym^{2}(\RR^{3})$ of symmetric second-order tensors on $\RR^{3}$. The action is written $g\star \ba:=g \ba g^{t}$ and there are three different symmetry classes (orthotropic $[\DD_{2}]$, transversely isotropic $[\OO(2)]$ and isotropic $[\SO(3)]$, see~\autoref{sec:Ela-symmetry-groups} for group definitions). The three corresponding symmetry strata $\Sigma_{[\DD_{2}]}$, $\Sigma_{[\OO(2)]}$ and $\Sigma_{[\SO(3)]}$, are characterized by polynomial equations. These conditions can be formulated, as algebraic equations involving either polynomial invariants, or \emph{polynomial covariants}~\cite{KP2000}.

Each second-order tensor $\ba\in \Sym^{2}(\RR^{3})$ splits as $\ba=\ba'+\frac{1}{3}(\tr\ba) \bq$, where
the deviatoric part $\ba'$ is a polynomial (linear) covariant of $\ba$, meaning that $\ba'$ is expressed polynomially (linearly) in the $a_{ij}$, and that for any $g\in \SO(3)$,
\begin{equation*}
  (g\star \ba)'=g\star \ba'.
\end{equation*}
A less common but very important polynomial covariant of $\ba$ was obtained in~\cite{OKDD2021} using the generalized cross product \eqref{eq:symmetric-tensor-product},
\begin{equation*}
  \bS(\ba):=\ba\times \ba^{2}\in \Sym^{3}(\RR^{3}),
  \quad \textrm{with} \quad
  g\star \left(\ba\times \ba^{2}\right)=(g\star\ba)\times (g\star\ba)^{2},
\end{equation*}
for any rotation $g$.

The algebraic equations characterizing each symmetry stratum of $\Sym^{2}(\RR^{3})$ are stated in table~\ref{tab:H2_Sym_Class_Equations}, where we consider the three following polynomial invariants
\begin{equation}\label{eq:Int_Basis_Sym2}
  I_{1}:=\tr\ba,\quad J_{2}:=\tr (\ba^{\prime \, 2}), \qquad J_{3}:=\tr (\ba^{\prime \, 3}),
\end{equation}
which constitute a minimal integrity basis of $\RR[\Sym^{2}(\RR^{3})]^{\SO(3)}$.

\begin{rem}
  The characterization conditions using covariants are of degree (in $\ba$) half the degree of those using invariants. Indeed
  \begin{equation*}
    J_{2}= \norm{\ba'}^{2}, \qquad J_{2}^{3}-6J_{3}^{2} = 12\norm{\ba\times \ba^{2}}^{2}.
  \end{equation*}
\end{rem}

\begin{table}[h]
  \renewcommand{\arraystretch}{1.2}
  \begin{tabular}{|c||c|c|}
    \hline
    Stratum            & Conditions in terms of invariants         & Conditions in terms of covariants        \\
    \hline \hline
    $\strata{\DD_{2}}$ & $J_{2}^{3}-6J_{3}^{2}\ne 0$               & $\ba\times\ba^{2}\ne \bz$                \\
    $\strata{\OO(2)}$  & $J_{2}^{3}-6J_{3}^{2}=0$ and $J_{2}\ne 0$ & $\ba\times\ba^{2}=\bz$ and $\ba'\ne \bz$ \\
    $\strata{\SO(3)}$  & $J_{2}=0$                                 & $\ba'=\bz$                               \\
    \hline
  \end{tabular}
  \caption{Algebraic equations defining the symmetry strata of $\Sym^{2}(\RR^{3})$~\cite{OKDD2021}.}
  \label{tab:H2_Sym_Class_Equations}
\end{table}

In contrast to the entire orbit space $\SV/G$, each orbit space $\Sigma_{[H]}/G$ is a \emph{smooth manifold}~\cite{AS1983,Bre1972,Olver1995} and when $\SV = \Sym^{2}(\RR^{3})$ we have:
\begin{equation*}
  \dim(\Sigma_{[\DD_{2}]}/\SO(3))=3,\quad \dim(\Sigma_{[\OO(2)]}/\SO(3))=2,\quad \dim(\Sigma_{[\SO(3)]}/\SO(3))=1.
\end{equation*}

Next, we will show how theorem~\ref{thm:Key_Thm_Separating_Set} helps us to obtain minimal functional bases for the orthotropic ($\Sigma_{[\DD_{2}]}$) and the transversely isotropic ($\Sigma_{[\OO(2)]}$) strata.

\subsection{Orthotropic stratum}

The orbit space $\Sigma_{[\DD_{2}]}/\SO(3)$ is three dimensional. An integrity basis is also a separating set~\cite[Appendix C]{AS1983}, and by the Wineman--Pipkin theorem~\ref{thm:sepimpliesfunc}, it is also a functional basis. Thus, the set~\eqref{eq:Int_Basis_Sym2}, satisfying the hypotheses of theorem~\ref{thm:Key_Thm_Separating_Set}, is an example of application of this theorem, which is formulated below.

\begin{lem}
  A \emph{minimal} functional basis for $\strata{\DD_{2}}$, \emph{i.e.}, for orthotropic second-order tensors, consists of the three polynomial invariants
  \begin{equation*}
    \kappa_{1}:=I_{1}=\tr\ba,\qquad \kappa_{2}:=J_{2}=\tr (\ba^{\prime\, 2}), \qquad \kappa_{3}:=J_{3}=\tr (\ba^{\prime\, 3}).
  \end{equation*}
\end{lem}

\subsection{Transversely isotropic stratum}
\label{sec:aIT}

In this case, we first note that a second-order tensor $\ba$ is in the symmetry stratum $\strata{\OO(2)}$ if and only if there exists a rotation $g\in \SO(3)$ such that $\ba=g\star \ba_0$, where $\ba_0$ is
\begin{equation}\label{eq:Form_For_Sym2_IsoT}
  \ba_0=\begin{pmatrix}
    \delta_{1}-\delta_{2} & 0                     & 0                      \\
    0                     & \delta_{1}-\delta_{2} & 0                      \\
    0                     & 0                     & \delta_{1}+2\delta_{2}
  \end{pmatrix},\quad \delta_{2}\ne 0,
\end{equation}
in the orthonormal basis $(\ee_{i})$. The condition $\delta_{2}\ne 0$ means that $\ba_0$ is genuinely transversely isotropic (and not isotropic). Moreover its symmetry group is the subgroup $\OO(2)$ of $\SO(3)$ defined in~\autoref{sec:Ela-symmetry-groups}.

\begin{lem}\label{lem:H2_IT}
  A \emph{minimal} functional basis for $\strata{\OO(2)}$, \emph{i.e.}, for transversely isotropic symmetric second-order tensors, consists of the two rational invariants
  \begin{equation*}
    \kappa_{1}:=I_{1}, \quad \kappa_{2}:=\displaystyle \frac{J_{3}}{J_{2}}.
  \end{equation*}
\end{lem}

\begin{proof}
  Evaluating the invariants $J_{2}$ and $J_{3}$ on~\eqref{eq:Form_For_Sym2_IsoT}, we get
  \begin{equation*}
    J_{2}(\ba)=6\delta_{2}^{2},\quad J_{3}(\ba)=6\delta_{2}^{3}, \quad \delta_{2} \ne 0,
  \end{equation*}
  and hence $\kappa_{2}(\ba)=\delta_{2}$. We have therefore
  \begin{equation*}
    I_{1}(\ba) = \kappa_{1}(\ba), \qquad J_{2}(\ba)=6\kappa_{2}^{2}(\ba), \qquad J_{3}(\ba)=6\kappa_{2}^{3}(\ba),
  \end{equation*}
  and the result follows by theorem~\ref{thm:Key_Thm_Separating_Set} applied to $\SV=\Sym^{2}(\RR^{3})$ and the symmetry stratum $\Sigma_{[\OO(2)]}$, with $\dim\left(\Sigma_{[\OO(2)]}/\SO(3)\right)=2$.
\end{proof}

The rational invariants $\kappa_{1},\kappa_{2}$ in lemma~\ref{lem:H2_IT} can be considered as \emph{global parameters} of $\SX=\Sigma_{[\OO(2)]}/\SO(3)$.
\begin{prop}\label{prop:devaIT}
  Any transversely isotropic second-order symmetric tensor $\ba\in \strata{\OO(2)}$ can be written
  \begin{equation}\label{eq:devaIT}
    \ba=\frac{1}{3}\kappa_{1}\bq+ 3 \kappa_{2}\, \bt, \qquad
    \kappa_{1}:=I_{1},  \qquad \kappa_{2}:= \frac{J_{3}}{J_{2}}= \frac{\sign(J_{3})}{\sqrt{6}} \norm{\ba'},
  \end{equation}
  with $\bt:=(\nn\otimes \nn)'$, $\norm{\nn} =1$, where the vector $\nn$ defines the axis of transverse isotropy and $\sign(x)=x/\abs{x}$ is the sign function.
\end{prop}

\section{Functional bases on symmetry strata of harmonic fourth-order tensors}
\label{sec:H4}

Let us now consider the vector space of \emph{fourth-order harmonic tensors} in $\RR^{3}$
\begin{equation*}
  \HH^{4}(\RR^{3}) := \set{\bH\in \Sym^{4}(\RR^{3}),\quad  \tr \bH=0},
\end{equation*}
\emph{i.e.}, of traceless totally symmetric fourth-order tensors. It is of dimension nine and appears as an irreducible subspace in the harmonic decomposition of $\Ela$. Its structure is more tricky than the one of $\HH^{2}(\RR^{3})$ and has been investigated in~\cite{FV1996} and~\cite{AKP2014,OKDD2021,OKDD2018}. The eight symmetry classes $[H]$ for $\HH^{4}(\RR^{3})$ are the same as for $\Ela$ (see~\autoref{fig:bifurcationEla},~\autoref{sec:stratification-H4}). Each orbit space $\Sigma_{[H]}/\SO(3)$ is a smooth manifold, and (see~\cite{AKP2014}, for instance)
\begin{equation*}
  \begin{aligned}
     & \dim(\Sigma_{[\triv]}/\SO(3))=6,   &  & \dim(\Sigma_{[\ZZ_{2}]}/\SO(3))=5, &  & \dim(\Sigma_{[\DD_{2}]}/\SO(3))=3,
    \\
     & \dim(\Sigma_{[\DD_{4}]}/\SO(3))=2, &  & \dim(\Sigma_{[\DD_{3}]}/\SO(3))=2, &  & \dim(\Sigma_{[\OO(2)]}/\SO(3))=1,
    \\
     & \dim(\Sigma_{[\octa]}/\SO(3))=1,   &  & \dim(\Sigma_{[\SO(3)]}/\SO(3))=0.
  \end{aligned}
\end{equation*}

A minimal integrity basis of nine polynomial invariants for the invariant algebra of $\HH^{4}(\RR^{3})$, has been derived by Boehler, Kirillov and Onat~\cite{BKO1994}, using previous works on binary forms by Shioda~\cite{Shi1967}. An alternative minimal integrity basis has been proposed in~\cite[Theorem 2.7]{DADKO2019}.
It involves only the two second-order covariants $\bd_{2}$ and $\bd_{3}$~\cite{BKO1994}
\begin{equation*}
  \bd_{2} := \tr_{13} \bH^{2}, \qquad \bd_{3} := \tr_{13} \bH^{3},
\end{equation*}
which, in components write
\begin{equation*}
  (\bd_{2})_{ij} = H_{ipqr}H_{pqrj}, \quad \text{and} \quad (\bd_{3})_{ij} = H_{ikpq}H_{pqrs}H_{rskj}.
\end{equation*}

Here, we shall work with a slightly modified integrity basis,
\begin{equation} \label{eq:Ikinvariants}
  \begin{array} {lll}
    I_{2} := \tr \bd_{2},                                    & I_{3} := \tr \bd_{3},                & I_{4} := \tr {\bd_{2}^{\prime}}^{2},
    \\
    I_{5} := \tr (\bd_{2}^{\prime} \bd_{3}^{\prime}),        & I_{6} := \tr {\bd_{2}^{\prime}}^{3}, & I_{7} := \tr ({\bd_{2}^{\prime}}^{2}\bd_{3}^{\prime}),
    \\
    I_{8} := \tr (\bd_{2}^{\prime} {\bd_{3}^{\prime}}^{2}) , & I_{9} := \tr {\bd_{3}^{\prime}}^{3}, & I_{10} := \tr ({\bd_{2}^{\prime}}^{2} {\bd_{3}^{\prime}}^{2}).
  \end{array}
\end{equation}

In the following, we consider Kelvin's representation of a fourth-order harmonic tensor $\bH=(H_{ijkl})$, \emph{i.e.}, in an orthonormal basis, the symmetric matrix \cite{TKel1878,Ryc1984}
\begin{equation*}
  [\bH] :=
  \begin{pmatrix}
    H_{1111}         & H_{1122}         & H_{1133}         & \sqrt{2}H_{1123} & \sqrt{2}H_{1113} & \sqrt{2}H_{1112} \\
    H_{1122}         & H_{2222}         & H_{2233}         & \sqrt{2}H_{2223} & \sqrt{2}H_{1223} & \sqrt{2}H_{1222} \\
    H_{1133}         & H_{2233}         & H_{3333}         & \sqrt{2}H_{2333} & \sqrt{2}H_{1333} & \sqrt{2}H_{1233} \\
    \sqrt{2}H_{1123} & \sqrt{2}H_{2223} & \sqrt{2}H_{2333} & 2H_{2233}        & 2H_{1233}        & 2H_{1223}        \\
    \sqrt{2}H_{1113} & \sqrt{2}H_{1223} & \sqrt{2}H_{1333} & 2H_{1233}        & 2H_{1133}        & 2H_{1123}        \\
    \sqrt{2}H_{1112} & \sqrt{2}E_{1222} & \sqrt{2}H_{1233} & 2H_{1223}        & 2H_{1123}        & 2H_{1122}
  \end{pmatrix}
\end{equation*}
with 9 ($=\dim \HH^{4}(\RR^{3}))$ independent components since
\begin{align*}
   & H_{1111}  = - H_{1122} - H_{1133}, &                                    & H_{2222}  = -H_{1122}-H_{2233},
   &                                    & H_{3333}  = - H_{1133} - H_{2233},
  \\
   & H_{2333}  = - H_{1123} - H_{2223}, &                                    & H_{1113}  = - H_{1223} - H_{1333}, &  & H_{1222}  = - H_{1112} - H_{1233}.
\end{align*}

\subsection{Cubic stratum}

A fourth-order tensor $\bH\in \HH^{4}(\RR^{3})$ is at least cubic if and only if there exists a rotation $g\in \SO(3)$ such that $\bH=g\star \bH_{\octa}$, where $\bH_{\octa}$ has the following Kelvin representation~\cite{AKP2014},
\begin{equation}
  \label{eq:cubic_case}
  [\bH_{\octa}]= \left(
  \begin{array}{cccccc}
      8\delta  & -4\delta & -4\delta & 0        & 0        & 0        \\
      -4\delta & 8\delta  & -4\delta & 0        & 0        & 0        \\
      -4\delta & -4\delta & 8\delta  & 0        & 0        & 0        \\
      0        & 0        & 0        & -8\delta & 0        & 0        \\
      0        & 0        & 0        & 0        & -8\delta & 0        \\
      0        & 0        & 0        & 0        & 0        & -8\delta
    \end{array}
  \right),
\end{equation}
and $\bH_{\octa}$ is cubic if and only if $\delta\ne 0$.
The evaluation of the invariants~\eqref{eq:Ikinvariants} on~\eqref{eq:cubic_case} is
\begin{equation}\label{eq:Jkcubic}
  I_{2}(\bH) = 480\delta^{2}, \qquad I_{3}(\bH) = 1920\delta^{3}, \qquad I_{k}(\bH) = 0
  \quad \textrm{for $k=4$ to $10$.}
\end{equation}

\begin{prop}\label{prop:SepHcubic}
  A minimal functional basis for $\strata{\octa}$, \emph{i.e.}, for cubic fourth-order harmonic tensors $\bH\in \HH^{4}(\RR^{3})$, is reduced to the single rational invariant $\kappa := {I_{3}}/{I_{2}}$.
\end{prop}

\begin{proof}
  This is a direct consequence of theorem~\ref{thm:Key_Thm_Separating_Set} applied to $\SV=\HH^{4}(\RR^{3})$ and the cubic stratum $\strata{\octa}$ (of dimension 1). Indeed, we have $\kappa(\bH)=4\delta\ne 0$ for all $\bH\in \strata{\octa}$, and thus
  $I_{2}(\bH)=30\kappa^{2}(\bH)$ and $I_{3}(\bH)=30\kappa^{3}(\bH)$.
\end{proof}

\subsection{Transversely isotropic stratum}

A fourth-order tensor $\bH\in \HH^{4}(\RR^{3})$ is at least transversely isotropic if and only if there exists a rotation $g\in \SO(3)$ such that $\bH=g\star \bH_{\OO(2)}$, where $\bH_{\OO(2)}$ has the following Kelvin representation~\cite{AKP2014},
\begin{equation}\label{eq:transversly_{i}sotropic_case}
  [ \bH_{\OO(2)}]= \left(
  \begin{array}{cccccc}
      3\,\delta  & \delta     & -4\,\delta & 0          & 0          & 0         \\
      \delta     & 3\,\delta  & -4\,\delta & 0          & 0          & 0         \\
      -4\,\delta & -4\,\delta & 8\,\delta  & 0          & 0          & 0         \\
      0          & 0          & 0          & -8\,\delta & 0          & 0         \\
      0          & 0          & 0          & 0          & -8\,\delta & 0         \\
      0          & 0          & 0          & 0          & 0          & 2\,\delta
    \end{array}
  \right),
\end{equation}
and $\bH_{\OO(2)}$ is transversely isotropic if and only if $\delta\ne 0$.
The evaluation of the invariants~\eqref{eq:Ikinvariants} on~\eqref{eq:transversly_{i}sotropic_case} is
\begin{equation}\label{eq:JkIT}
  \begin{aligned}
     & I_{2} = 280\delta^{2},     &  & I_{3} = 720\delta^{3},               &  & I_{4} = \frac{20000}{3} \delta^{4},
    \\
     & I_{5} = 40000\delta^{5},   &  & I_{6} = \frac{2000000}{9}\delta^{6}, &  & I_{7} = \frac{4000000}{3}\delta^7,
    \\
     & I_{8} = 8000000\delta^{8}, &  & I_{9} = 48000000\delta^9,            &  & I_{10} = 800000000 \delta^{10}.
  \end{aligned}
\end{equation}
Observe that
\begin{equation*}
  \delta= \frac{7}{18} \frac{I_3}{I_2}
  = \frac{27}{250} \frac{I_4}{I_3}
  = \frac{1}{6} \frac{I_5}{I_4}
  = \frac{9}{50} \frac{I_6}{I_5}
  = \frac{1}{6} \frac{I_7}{I_6}
  = \frac{1}{6} \frac{I_8}{I_7}
  = \frac{1}{6} \frac{I_9}{I_8}
  = \frac{3}{50} \frac{I_{10}}{I_9},
\end{equation*}
is a rational invariant.
Following the same proof as for proposition~\ref{prop:SepHcubic}, we obtain the following result.

\begin{prop}\label{prop:SepHisoT}
  A minimal functional basis for $\strata{\OO(2)}$, \emph{i.e.}, for transversely isotropic fourth-order harmonic tensors $\bH\in \HH^{4}(\RR^{3})$, is reduced to the single rational invariant $\kappa := \delta$.
\end{prop}

\subsection{Tetragonal stratum}

A fourth-order tensor $\bH\in \HH^{4}(\RR^{3})$ is at least tetragonal if and only if there exists a rotation $g\in \SO(3)$ such that $\bH=g\star \bH_{\DD_{4}}$ where $\bH_{\DD_{4}}$ has the following Kelvin representation,
\begin{equation}\label{eq:tetragonal_case}
  [\bH_{\DD_{4}}]= \left(
  \begin{array}{cccccc}
      3\,\delta  -\sigma & \sigma+\delta     & -4\,\delta & 0          & 0          & 0                   \\
      \sigma+\delta      & 3\,\delta -\sigma & -4\,\delta & 0          & 0          & 0                   \\
      -4\,\delta         & -4\,\delta        & 8\,\delta  & 0          & 0          & 0                   \\
      0                  & 0                 & 0          & -8\,\delta & 0          & 0                   \\
      0                  & 0                 & 0          & 0          & -8\,\delta & 0                   \\
      0                  & 0                 & 0          & 0          & 0          & 2\,\sigma+2\,\delta
    \end{array}
  \right),
\end{equation}
and $\bH_{\DD_{4}}$ is tetragonal if and only if
$\sigma\ne 0$ and $\sigma^{2}-25\delta^{2}\ne 0$.
Recall here the following bifurcation conditions~\cite{AKP2014}: $(i)$ $\sigma=0$ implies transverse isotropy,
$(ii)$ $\sigma^{2}-25\delta^{2}= 0$ implies cubic symmetry, and $(iii)$ $\sigma = 0$ and $\delta = 0$ imply isotropy.
The evaluation of the invariants~\eqref{eq:Ikinvariants} on \eqref{eq:tetragonal_case} is
\begin{equation}\label{eq:Jktetra}
  \begin{aligned}
    I_{2} & = 8 (35\delta^{2} + \sigma^{2}),                 & I_{3} & = 48\delta (15\delta^{2} +\sigma^{2}),           & I_{4}  & = \frac{32}{3} (25\delta^{2} -\sigma^{2})^{2},
    \\
    I_{5} & = 64\delta (25\delta^{2} - \sigma^{2})^{2},      & I_{6} & = \frac{128}{9} (25\delta^{2} - \sigma^{2})^{3}, & I_{7}  & = \frac{256}{3}\delta (25\delta^{2} - \sigma^{2})^{3},
    \\
    I_{8} & = 512\delta^{2} (25\delta^{2} - \sigma^{2})^{3}, & I_{9} & =3072\delta^{3} (25\delta^{2} - \sigma^{2})^{3}, & I_{10} & = 2048\delta^{2} (25\delta^{2} - \sigma^{2})^{4}.
  \end{aligned}
\end{equation}
In accordance with remark~\ref{rem:J2K4}, $I_{4}\ne 0$ for all $\bH\in \strata{\DD_{4}}$.

\begin{prop}\label{prop:SepHtetra}
  A minimal functional basis for $\strata{\DD_{4}}$, \emph{i.e.}, for tetragonal fourth-order harmonic tensors $\bH\in \HH^{4}(\RR^{3})$, consists of the two rational invariants
  \begin{equation}\label{eq:coeff-tetragonal}
    \kappa_{1}:=\frac{I_{5}}{I_{4}},\qquad \kappa_{2}:=I_{2}.
  \end{equation}
\end{prop}

\begin{proof}
  For each $\bH\in \strata{\DD_{4}}$, we deduce by~\eqref{eq:Jktetra} that
  \begin{equation*}
    \delta = \frac{1}{6}\frac{I_{5}}{I_{4}} = \frac{1}{6}\kappa_{1},\qquad \sigma^{2} = \frac{1}{8} I_{2} - 35 \delta^{2}
    = \frac{1}{8} \kappa_{2} - \frac{35}{36}{\kappa_{1}}^{2}.
  \end{equation*}
  Since each $I_{k}$ ($2 \le k \le 10$) depends only on $\delta$ and $\sigma^{2}$, we deduce that they are functions of $\kappa_{1},\kappa_{2}$, and the proposition follows by theorem~\ref{thm:Key_Thm_Separating_Set}, since $\dim \strata{\DD_{4}}/\SO(3) = 2$~\eqref{eq:dimElastrata}.
\end{proof}

\begin{rem}\label{I2I3_nor_I3I4_tetra}
  Neither $\set{I_{2},I_{3}}$, nor $\set{I_{3},I_{4}}$ are separating sets. Indeed,
  \begin{itemize}
    \item for both tetragonal tensors $(\delta=1, \sigma=\sqrt{60})$ and $(\delta=3/2, \sigma=\sqrt{65/4})$, we have $I_{2}=760$ and $I_{3}=3600$, but they have different values for $I_{4}$.

    \item for both tetragonal tensors $(\delta=1, \sigma=\sqrt{63})$ and $(\delta=3/2, \sigma=\sqrt{73/4})$, we have $I_{3}=3744$ and $I_{4}=46208/3$, but they have different values for $I_{2}$.
  \end{itemize}
\end{rem}

\subsection{Trigonal stratum}

A fourth-order tensor $\bH\in \HH^{4}(\RR^{3})$ is at least trigonal if and only if there exists a rotation $g\in \SO(3)$ such that $\bH=g\star \bH_{\DD_{3}}$ where $\bH_{\DD_{3}}$ has the following Kelvin representation,
\begin{equation}\label{eq:trigonal_case}
  [\bH_{\DD_{3}}]= \left(
  \begin{array}{cccccc}
      3\,\delta       & \delta         & -4\,\delta & -\sqrt{2}\sigma & 0          & 0          \\
      \delta          & 3\,\delta      & -4\,\delta & \sqrt{2}\sigma  & 0          & 0          \\
      -4\,\delta      & -4\,\delta     & 8\,\delta  & 0               & 0          & 0          \\
      -\sqrt{2}\sigma & \sqrt{2}\sigma & 0          & -8\,\delta      & 0          & 0          \\
      0               & 0              & 0          & 0               & -8\,\delta & -2\,\sigma \\
      0               & 0              & 0          & 0               & -2\,\sigma & 2\,\delta
    \end{array}
  \right),
\end{equation}
and $\bH_{\DD_{3}}$ is trigonal if and only if
$\sigma\ne 0$ and $\sigma^{2}-50\delta^{2}\ne 0$.
Recall also the bifurcation conditions~\cite{AKP2014}: $(i)$ $\sigma=0$ implies transverse isotropy, $(ii)$ $\sigma^{2}-50\delta^{2}=0$ implies cubic symmetry, and $(iii)$ $\sigma = 0$ and $\delta = 0$ imply isotropy.
The evaluation of the invariants~\eqref{eq:Ikinvariants} on~\eqref{eq:trigonal_case} is
\begin{equation}\label{eq:Jktrigo}
  \begin{aligned}
     & I_{2} = 8 (35\delta^{2} + 2\sigma^{2}),            &  & I_{3} =144\delta (5\delta^{2} -\sigma^{2}),           &  & I_{4} = \frac{8}{3} (50\delta^{2} -\sigma^{2})^{2},      \\
     & I_{5}=16\delta (50\delta^{2} -\sigma^{2})^{2},     &  & I_{6} =  \frac{16}{9} (50\delta^{2} -\sigma^{2})^{3}, &  & I_{7}=\frac{32}{3}\delta (50\delta^{2} -\sigma^{2})^{3}, \\
     & I_{8}=64\delta^{2} (50\delta^{2} -\sigma^{2})^{3}, &  & I_{9}=384\delta^{3} (50\delta^{2} -\sigma^{2})^{3}    &  & I_{10}=  128\delta^{2} (50\delta^{2} -\sigma^{2})^{4}.
  \end{aligned}
\end{equation}
As for the tetragonal case, we have $I_{4}\ne 0$ for all $\bH\in \strata{\DD_{3}}$. Now, following the same proof as the one of proposition~\ref{prop:SepHtetra}, we get:

\begin{prop}\label{prop:SepHtrigo}
  A minimal functional basis for $\strata{\DD_{3}}$, \emph{i.e.}, for trigonal harmonic fourth-order tensors $\bH\in \HH^{4}(\RR^{3})$, consists of the two rational  invariants
  \begin{equation}\label{eq:coeff-trigonal}
    \kappa_{1}:=\frac{I_{5}}{I_{4}},\qquad \kappa_{2}:=I_{2}.
  \end{equation}
\end{prop}

\begin{rem}\label{I2I3_nor_I3I4_trigo}
  Neither $\set{I_{2}, I_{3}}$ nor $\set{I_{3}, I_{4}}$ are separating sets. Indeed,
  \begin{itemize}
    \item for both trigonal tensors $(\delta=1, \sigma=\sqrt{715/8})$ and $(\delta=3/2, \sigma=\sqrt{135/2})$, we have $I_{2}=1710$ and $I_{3}=-12150$, but they have different values for $I_{4}$.

    \item for both trigonal tensors $(\delta=1, \sigma=\sqrt{371/4})$ and $(\delta=3/2, \sigma=\sqrt{279/4})$, we have $I_{3}=-12636$ and $I_{4}=9747/2$, but they have different values for $I_{2}$.
  \end{itemize}
\end{rem}

We point out here that each proposed minimal functional basis concerns an \emph{exact symmetry stratum}. The proposed functional basis $\set{ \kappa_{1}=I_{5}/I_{4}, \kappa_{2}=I_{2}}$ happens to be identical for the tetragonal and trigonal strata. A natural question then arises: does this set remain a functional basis for the union of strata $\strata{\DD_{3}} \cup \strata{\DD_{4}}$? The answer is no as detailed in the following remark.

\begin{rem}\label{rem:non-sep-closed-stratum}
  By proposition \ref{prop:SepHtetra}, two tetragonal harmonic fourth-order tensors having the same values for $\kappa_{1}$ and $\kappa_{2}$ are indeed in the same orbit (as a functional basis is a separating set). The same holds,  by proposition \ref{prop:SepHtrigo}, if one considers two trigonal harmonic fourth-order tensors having the same values for $\kappa_{1}$ and $\kappa_{2}$. There exists, however, trigonal tensors that have the same value for $\kappa_{1}$ and $\kappa_{2}$ as some tetragonal tensors. Since, they are not on the same orbit as they do not belong to the same symmetry class, and the set $\set{\kappa_{1},\kappa_{2}}$ is not a functional basis for $\strata{\DD_{3}}\cup \strata{\DD_{4}}$.
\end{rem}

\subsection{Orthotropic stratum}

A fourth-order tensor $\bH\in \HH^{4}(\RR^{3})$ is at least orthotropic if and only if there exists a rotation $g\in \SO(3)$ such that $\bH=g\star \bH_{\DD_{2}}$ where $\bH_{\DD_{2}}$ has the following Kelvin representation~\cite{OKDD2018},
\begin{equation}\label{eq:ortho_case}
  [\bH_{\DD_{2}}]= \left(
  \begin{array}{cccccc}
      \lambda_{2}+\lambda_{3} & -\lambda_{3}            & -\lambda_{2}\           & 0             & 0             & 0             \\
      -\lambda_{3}            & \lambda_{3}+\lambda_{1} & -\lambda_{1}            & 0             & 0             & 0             \\
      -\lambda_{2}            & -\lambda_{1}            & \lambda_{1}+\lambda_{2} & 0             & 0             & 0             \\
      0                       & 0                       & 0                       & -2\lambda_{1} & 0             & 0             \\
      0                       & 0                       & 0                       & 0             & -2\lambda_{2} & 0             \\
      0                       & 0                       & 0                       & 0             & 0             & -2\lambda_{3}
    \end{array}
  \right),
\end{equation}
and $\bH_{\DD_{2}}$ is orthotropic if and only if $\lambda_{1},\lambda_{2},\lambda_{3}$ are all distinct. In fact, setting
\begin{equation*}
  \Delta:=(\lambda_{1}-\lambda_{2})(\lambda_{2}-\lambda_{3})(\lambda_{1}-\lambda_{3}),
\end{equation*}
we have by direct evaluation of the invariant $\norm{\tr(\bH\times \bd_{2})}^{2}$ on~\eqref{eq:ortho_case}:
\begin{equation}\label{eq:TrHd2_et_Disc}
  \norm{\tr(\bH\times \bd_{2})}^{2} = \frac{6}{25}\Delta^{2} .
\end{equation}

The evaluation of the integrity basis $\set{I_{2},\dotsc,I_{10}}$ of $\bH$ on~\eqref{eq:ortho_case} can be expressed polynomially using the elementary symmetric functions \cite[section 5.5]{AKP2014}
\begin{equation*}
  \sigma_{1}:=\lambda_{1}+\lambda_{2}+\lambda_{3}, \qquad
  \sigma_{2}:=\lambda_{1}\lambda_{2}+\lambda_{2}\lambda_{3}+\lambda_{2}\lambda_{3}, \qquad
  \sigma_{3}:=\lambda_{1}\lambda_{2}\lambda_{3}.
\end{equation*}
Conversely, the $\sigma_{i}$ can be expressed rationally in the $I_{k}$.

\begin{prop}\label{prop:SepHortho}
  A minimal functional basis for $\strata{\DD_{2}}$, \emph{i.e.}, for orthotropic harmonic fourth-order tensors $\bH\in \HH^{4}(\RR^{3})$, consists of the three rational invariants
  \begin{equation}\label{eq:kappak-ortho}
    \begin{split}
      \sigma_{1} &:=\displaystyle \frac{1}{96}
      \frac{6 I_{7}+3 I_{3} I_{4}-2 I_{2} I_{5}}{\Delta^2},
      \\
      \sigma_{2} &:= \displaystyle\frac{4}{7}\sigma_{1}^{\, 2}-\frac{1}{14}I_{2},
      \\
      \sigma_{3}& :=\displaystyle  \frac{1}{7} \sigma_{1}^{\, 3}-\frac{1}{56}\sigma_{1}I_{2}-\frac{1}{24}I_{3} ,
    \end{split}
  \end{equation}
  where $\Delta^{2}= \frac{1}{1296}\left(2 {I_{2}}^{3}-60 {I_{3}}^{2}-9 I_{2} I_{4} + 18 I_{6}\right)\ne 0$,
\end{prop}

\begin{proof}
  For each $\bH\in \strata{\DD_{2}}$, we can write $\bH=g\star \bH_{\DD_{2}}$ where $\bH_{\DD_{2}}$ is given by~\eqref{eq:ortho_case}. Now, a direct evaluation on the normal form \eqref{eq:ortho_case} leads to
  \begin{equation*}
    6 I_{7}+3 I_{3} I_{4}-2 I_{2} I_{5}=96\sigma_{1}\Delta^{2},\quad
    2 I_{2}^{\, 3}-60 I_{3}^{\,2}-9 I_{2} I_{4}+18 I_{6}=1296\Delta^{2}.
  \end{equation*}
  Hence, we obtain the first equation of~\eqref{eq:kappak-ortho}, while the others are obtained in the same way. Finally, each invariant $I_{2},\dotsc,I_{10}$ is a polynomial function of $\sigma_{1},\sigma_{2},\sigma_{3}$~\cite{AKP2014}, and, since $\dim\strata{\DD_{2}}/\SO(3) = 3$, the conclusion follows by theorem~\ref{thm:Key_Thm_Separating_Set}.
\end{proof}

\section{Functional bases on symmetry strata of elasticity tensors}
\label{sec:Ela}

We finally address the problem of the determination of minimal functional bases for the symmetry strata of the elasticity tensor (except for the orthotropic $\strata{\DD_{2}}$, the monoclinic $\strata{\ZZ_{2}}$ and the triclinic $\strata{\triv}$ strata, which will be investigated in a future work).
The isotropic case is trivial, a minimal functional basis for the isotropic stratum $\strata{\SO(3)}$ consists of the two Lamé coefficients. The cubic case is straightforward and treated in~section \ref{subsec:Ela-cubic}. In order to derive our results for the trigonal $\strata{\DD_{3}}$, tetragonal $\strata{\DD_{4}}$ and transversely isotropic $\strata{\OO(2)}$ strata, we shall define in~section \ref{subsec:ti-Ela-cov} a non vanishing second-order covariant $\bt=\bt(\bE)$ of $\bE$.

We recall the dimensions of the eight orbit spaces $\Sigma_{[H]}/\SO(3)$ (see~\cite{AKP2014}),
\begin{equation}\label{eq:dimElastrata}
  \begin{aligned}
     & \dim(\Sigma_{[\triv]}/\SO(3))    = 18, &  & \dim(\Sigma_{[\ZZ_{2}]}/\SO(3))  = 12, &  & \dim(\Sigma_{[\DD_{2}]}/\SO(3))  = 9,
    \\
     & \dim(\Sigma_{[\DD_{4}]}/\SO(3)) = 6,   &  & \dim(\Sigma_{[\DD_{3}]}/\SO(3)) = 6,   &  & \dim(\Sigma_{[\OO(2)]}/\SO(3))   = 5,
    \\
     & \dim(\Sigma_{[\octa]}/\SO(3))    = 3,  &  & \dim(\Sigma_{[\SO(3)]}/\SO(3))   = 2.  &  &
  \end{aligned}
\end{equation}

\subsection{Elasticity cubic stratum}
\label{subsec:Ela-cubic}

By theorem~\ref{thm:main}, an elasticity tensor
\begin{equation*}
  \bE=(\tr \bd, \tr \bv, \bd^{\prime}, \bv^{\prime}, \bH) \in \Ela
\end{equation*}
is cubic if and only if $\bd^{\prime}=\bv^{\prime}=\bd_{2}^{\prime}=0$ and $I_{2}(\bH)=\tr \bd_{2}\ne 0$ (meaning that $\bH\in \HH^{4}(\RR^{3})$ is cubic). Now, by proposition~\ref{prop:SepHcubic} and since $\dim(\Sigma_{[\octa]}/\SO(3))=3$, we have the following result.

\begin{thm}\label{thm:ECubic}
  Let $\bE=(\tr \bd, \tr \bv, 0, 0, \bH)$ be a cubic elasticity tensor. A minimal functional basis for $\strata{\octa}$ consists of the three rational invariants
  \begin{equation}\label{eq:baseEcubic}
    \kappa_{1}:=\tr \bd, \quad  \kappa_{2}:=\tr \bv, \quad
    \kappa_{3}:= \frac{I_{3}}{I_{2}}.
  \end{equation}
\end{thm}

\subsection{A transversely isotropic second-order covariant}
\label{subsec:ti-Ela-cov}

The goal, here, is to build a symmetric second-order covariant of $\bE \in \Ela$ which is strictly transversely isotropic for all trigonal, tetragonal and transversely isotropic tensors. Observe that each symmetric second-order covariant, $\bt(\bE)$, is necessarily at least transversely isotropic since it inherits the symmetries of $\bE$ and since a second-order symmetric tensor can only be either orthotropic (three distinct eigenvalues), transversely isotropic (two distinct eigenvalues) or isotropic (only one eigenvalue). It is however not obvious to find such a covariant which remains strictly transversely isotropic for all
\begin{equation*}
  \bE \in \strata{\DD_{3}} \cup \strata{\DD_{4}} \cup \strata{\OO(2)}.
\end{equation*}
To build such a covariant, we use corollary~\ref{cor:transversely-isotropic}, which forbids $(\bd^{\prime},\bv^{\prime},\bd_{2}^{\prime})$ to be isotropic, and denote by $\langle \nn \rangle$ the direction of transverse isotropy of the triplet $(\bd^{\prime},\bv^{\prime},\bd_{2}^{\prime})$. By proposition~\ref{prop:devaIT}, with $\norm{\nn}=1$, $\norm{(\nn \otimes \nn)'}=\sqrt{\frac{2}{3}}$, we get thus
\begin{equation*}
  \bd^{\prime}= \pm\sqrt{\frac{3}{2}} \norm{\bd^{\prime}} (\nn \otimes \nn)',
  \quad
  \bv^{\prime}=\pm \sqrt{\frac{3}{2}} \norm{\bv^{\prime}}  (\nn \otimes \nn)',
  \quad
  \bd_{2}^{\prime}=\pm \sqrt{\frac{3}{2}} \norm{\bd_{2}^{\prime}}  (\nn \otimes \nn)',
\end{equation*}
and
\begin{equation*}
  \norm{\bd^{\prime}}^{2}\, {\bd^{\prime}}^{2}+ \norm{\bv^{\prime}}^{2}\,{\bv^{\prime}}^{2}+{\bd_{2}^{\prime}}^{2}=\frac{3}{2}
  \left( \norm{\bd^{\prime}}^{4} +\norm{\bv^{\prime}}^{4} + \norm{\bd_{2}^{\prime}}^{2}\right)  (\nn \otimes \nn)^{\prime \, 2}.
\end{equation*}
The property $((\nn \otimes \nn)^{\prime \, 2})'=\frac{1}{3} (\nn \otimes \nn)'$ leads to
\begin{equation*}
  (\norm{\bd^{\prime}}^{2}\, {\bd^{\prime}}^{2}+ \norm{\bv^{\prime}}^{2}\,{\bv^{\prime}}^{2}+{\bd_{2}^{\prime}}^{2})'= \frac{1}{2}\left(\norm{\bd^{\prime}}^{4} +\norm{\bv^{\prime}}^{4}+\norm{\bd_{2}^{\prime}}^{2}\right)\, (\nn \otimes \nn)' \ne 0,
\end{equation*}
as $\norm{\bd^{\prime}}^{4} +\norm{\bv^{\prime}}^{4} + \norm{\bd_{2}^{\prime}}^{2}\ne 0$ over the entire union of strata $\strata{\OO(2)}\cup\strata{\DD_{3}} \cup\strata{\DD_{4}}$.

Therefore, this allows us to define the deviatoric second-order rational covariant
\begin{equation}\label{eq:t}
  \bt:=2 \frac{(\norm{\bd^{\prime}}^{2}\, {\bd^{\prime}}^{2}+ \norm{\bv^{\prime}}^{2}\,{\bv^{\prime}}^{2}+{\bd_{2}^{\prime}}^{2})'}{\norm{\bd^{\prime}}^{4} +\norm{\bv^{\prime}}^{4}+\norm{\bd_{2}^{\prime}}^{2}}\ne 0
\end{equation}
for every elasticity tensor which is either trigonal, tetragonal or transversely isotropic. It is normalized in such a way that
\begin{equation*}
  \bt= (\nn\otimes\nn)' , \qquad \norm{\nn}=1, \qquad \norm{\bt}=\sqrt{\frac{2}{3}},
\end{equation*}
and thus (since $\bt = \bt'$)
\begin{equation}\label{eq:dvd2t}
  \bd^{\prime}=\frac{3}{2}\left(\bd:\bt\right) \bt,
  \qquad
  \bv^{\prime}=\frac{3}{2}\left(\bv:\bt \right) \bt,
  \qquad
  \bd_{2}^{\prime}=\frac{3}{2}\left(\bd_{2}:\bt \right) \bt.
\end{equation}

\subsection{Elasticity transversely isotropic stratum}

By corollary~\ref{cor:transversely-isotropic}, if $\bE$ is transversely isotropic, then, the triplet $(\bd^{\prime}, \bv^{\prime}, \bd_{2}^{\prime})$ is transversely isotropic, and thus
\begin{equation*}
  \norm{\bd^{\prime}}^{4} + \norm{\bv^{\prime}}^{4} + \norm{\bd_{2}^{\prime}}^{2} \ne 0.
\end{equation*}

\begin{thm}\label{thm:EIsoT}
  Let $\bE=(\tr \bd, \tr \bv, \bd^{\prime}, \bv^{\prime}, \bH)\in\strata{\OO(2)}$ be a transversely isotropic elasticity tensor and
  \begin{equation}\label{eq:Iso_transvers_tensor_t}
    \bt= 2 \frac{(\norm{\bd^{\prime}}^{2}\, {\bd^{\prime}}^{2}+ \norm{\bv^{\prime}}^{2}\,{\bv^{\prime}}^{2}+{\bd_{2}^{\prime}}^{2})'}{\norm{\bd^{\prime}}^{4} +\norm{\bv^{\prime}}^{4}+\norm{\bd_{2}^{\prime}}^{2}}\in \HH^{2}(\RR^{3}).
  \end{equation}
  A minimal functional basis for $\strata{\OO(2)}$ consists of the five rational invariants
  \begin{equation}\label{eq:baseEisoT}
    \kappa_{1}:=\tr \bd, \quad  \kappa_{2}:=\tr \bv, \quad
    \kappa_{3}:=\bd:\bt,
    \quad
    \kappa_{4}:= \bv:\bt,
    \quad
    \kappa_{5}:=\bt:\bH:\bt.
  \end{equation}
\end{thm}

\begin{proof}
  Let $\bE$ and $\overline{\bE}$ be two transversely isotropic elasticity tensors with the same invariants $\kappa_{1},\dots,\kappa_{5}$. We have to show that there exists $g \in \SO(3)$, such that
  \begin{equation*}
    g \star \bd^{\prime} = \overline{\bd^{\prime}}, \qquad g \star \bv^{\prime} = \overline{\bv^{\prime}}, \qquad g \star \bH = \overline{\bH}.
  \end{equation*}
  Now, $\bt(\bE)$ being the covariant defined by~\eqref{eq:Iso_transvers_tensor_t}, we can write (see section~\ref{subsec:ti-Ela-cov})
  \begin{equation*}
    \bt=(\nn\otimes \nn)',\quad \overline{\bt}=(\overline{\nn}\otimes \overline{\nn})',
  \end{equation*}
  where $\nn$ and $\overline{\nn}$ are two unit vectors. Choose a rotation $g\in \SO(3)$ such that $g\nn=\overline{\nn}$. Then, we get $g\star \bt=\overline{\bt}$, and by~\eqref{eq:dvd2t} and proposition~\ref{prop:devaIT}
  \begin{equation*}
    \bd=\frac{\kappa_{1}}{3} \bq+\frac{3}{2}(\bd:\bt)\bt=\frac{\kappa_{1}}{3}\bq+\frac{3}{2}\kappa_{2}\,\bt\implies g\star \bd=\frac{\kappa_{1}}{3}\bq+\frac{3}{2}\kappa_{2}\, g\star\bt=\overline{\bd}.
  \end{equation*}
  The argumentation is the same for $\bv$ and $\overline{\bv}$. Finally, using the reconstruction formula~\eqref{eq:ReconstIT}, we have
  \begin{equation*}
    \bH=\frac{35}{8}\left(\bt:\bH:\bt\right)\bt\ast \bt=\frac{35}{8}\kappa_{5}\, \bt\ast \bt\implies g\star \bH=\frac{35}{8}\kappa_{5}\, (g\star\bt)\ast (g\star\bt)=\overline{\bH},
  \end{equation*}
  where the harmonic square $\bt \ast \bt$ is the fourth order harmonic part of $\bt \odot \bt=(\bt \otimes \bt)^s$,
  \begin{equation*}
    \bt \ast \bt:=\bt \odot \bt - \frac{4}{7}\, \bq \odot \bt^{2} + \frac{2}{35} \norm{\bt}^{2}\, \bq \odot \bq,
  \end{equation*}
  such as $g\star (\bt \ast \bt)= (g\star\bt)\ast (g\star\bt)$.
  This achieves the proof that $\set{\kappa_{1},\dotsc,\kappa_{5}}$ is a functional basis for $\Sigma_{[\OO(2)]}$ and the minimality follows by remark~\ref{rem:Minimality_and_Dimension}, since $\dim(\Sigma_{[\OO(2)]}/\SO(3)) = 5$.
\end{proof}

\subsection{Elasticity tetragonal stratum}

Given a tetragonal elasticity tensor
\begin{equation*}
  \bE=(\tr \bd, \tr \bv, \bd^{\prime}, \bv^{\prime}, \bH)\in \strata{\DD_{4}},
\end{equation*}
the triplet $(\bd^{\prime}, \bv^{\prime}, \bd_{2}^{\prime})$ is transversely isotropic (by corollary \ref{cor:transversely-isotropic}) and $\bH \in \HH^{4}(\RR^{3})$ is either cubic or tetragonal (it is neither isotropic, nor transversely isotropic~\cite{FV1996,OKDD2021}).

\begin{thm}\label{thm:Etetra}
  Let $\bE=(\tr \bd, \tr \bv, \bd^{\prime}, \bv^{\prime}, \bH)$ be a tetragonal elasticity tensor and
  \begin{equation*}
    \bt= 2 \frac{(\norm{\bd^{\prime}}^{2}\, {\bd^{\prime}}^{2}+ \norm{\bv^{\prime}}^{2}\,{\bv^{\prime}}^{2}+{\bd_{2}^{\prime}}^{2})'}{\norm{\bd^{\prime}}^{4} +\norm{\bv^{\prime}}^{4}+\norm{\bd_{2}^{\prime}}^{2}}.
  \end{equation*}
  A minimal functional basis for $\strata{\DD_{4}}$ consists of the six rational invariants
  \begin{align*}
     & \kappa_{1}:=\tr \bd,
     &                       & \kappa_{2}:=\tr \bv,
     &                       & \kappa_{3}:=\bd:\bt,
    \\
     & \kappa_{4}:= \bv:\bt,
     &                       & \kappa_{5}:=\bt:\bH:\bt,
     &                       &
    \kappa_{6}:=I_{2}.
  \end{align*}
\end{thm}

\begin{rem}\label{rem:Etetra}
  In this set, $\kappa_{6}=I_{2}=\tr \bd_{2}$ can be replaced by $I_{3}=\tr \bd_{3}$ (by lemma~\ref{lem:Couple_tetra_tri}).
\end{rem}

\begin{proof}
  Let $\bE$ and $\overline{\bE}$ be two tetragonal elasticity tensors. Then, the pairs $(\bH,\bt)$ and $(\overline{\bH},\overline{\bt})$ are necessarily both tetragonal, since they have the same respective symmetry as $(\bd^{\prime}, \bv^{\prime}, \bH)$ and $(\overline{\bd^{\prime}}, \overline{\bv^{\prime}}, \overline{\bH})$. If they have the same invariants $\kappa_{1},\dotsc,\kappa_{6}$, then,
  \begin{equation*}
    \kappa_{5} = \bt:\bH:\bt=\overline{\bt}:\overline{\bH}:\overline{\bt},\quad \text{and} \quad \kappa_{6} = I_{2}(\bH)=I_{2}(\overline{\bH}),
  \end{equation*}
  and thus, by lemma~\ref{lem:Couple_tetra_tri}, $I_{k}(\bH)=I_{k}(\overline{\bH})$ for $2 \le k \le 10$. Hence, there exists $g\in \SO(3)$ such that $g\star \bH=\overline{\bH}$. Now, two cases can happen.
  \begin{enumerate}
    \item $\bH$ is tetragonal and has thus the same symmetry group as the pair $(\bH,\bt)$ (the same holds for $\overline{\bH}$ and $(\overline{\bH},\overline{\bt})$). In that case, let $\langle \nn \rangle$ be the principal axis of symmetry group of $\bH$, and $\langle \overline{\nn} \rangle$, the one for $\overline{\bH}$. Then, $g\nn = \pm \overline{\nn}$ and thus $g\star \bt=\overline{\bt}$.
    \item $\bH$ is cubic. Then, the principal axis $\langle \nn \rangle$ of the tetragonal pair $(\bH,\bt)$ is necessarily one of the three principal axes of the cubic tensor $\bH$ (and similarly for the pair $(\overline{\bH},\overline{\bt})$). Since $g$ sends each principal axis of $\bH$ onto a principal axis of $\overline{\bH}$, it is possible to change $g$ such that $g\nn = \overline{\nn}$, and thus that $g \star \bt = \overline{\bt}$ (keeping $g\star \bH=\overline{\bH}$), by replacing $g$ by $gh$, where $h$ belongs to the symmetry group of $\bH$ (see~\cite[Lemma 8.9]{OKDD2018} for details).
  \end{enumerate}
  In both cases, we conclude as in the proof of theorem~\ref{thm:EIsoT}, and the minimality follows since $\dim(\Sigma_{[\DD_{4}]}/\SO(3)) = 6$.
\end{proof}

\subsection{Elasticity trigonal stratum}
\label{sec:Ela-trig}

Given a trigonal elasticity tensor
\begin{equation*}
  \bE=(\tr \bd, \tr \bv, \bd^{\prime}, \bv^{\prime}, \bH)\in \strata{\DD_{3}},
\end{equation*}
the triplet $(\bd^{\prime}, \bv^{\prime}, \bd_{2}^{\prime})$ is transversely isotropic (by corollary \ref{cor:transversely-isotropic}) and $\bH \in \HH^{4}(\RR^{3})$ is either cubic or trigonal (it is neither isotropic, nor transversely isotropic~\cite{FV1996,OKDD2021}).
The proof of the following result is obtained in the same way as in the tetragonal case.

\begin{thm}\label{thm:Etrigo}
  Let $\bE=(\tr \bd, \tr \bv, \bd^{\prime}, \bv^{\prime}, \bH)$ be a trigonal elasticity tensor and
  \begin{equation*}
    \bt= 2 \frac{(\norm{\bd^{\prime}}^{2}\, {\bd^{\prime}}^{2}+ \norm{\bv^{\prime}}^{2}\,{\bv^{\prime}}^{2}+{\bd_{2}^{\prime}}^{2})'}{\norm{\bd^{\prime}}^{4} +\norm{\bv^{\prime}}^{4}+\norm{\bd_{2}^{\prime}}^{2}}.
  \end{equation*}
  A minimal functional basis for $\strata{\DD_{3}}$ consists of the six rational invariants
  \begin{align*}
     & \kappa_{1}:=\tr \bd,
     &                       & \kappa_{2}:=\tr \bv,
     &                       & \kappa_{3}:=\bd:\bt,
    \\
     & \kappa_{4}:= \bv:\bt,
     &                       & \kappa_{5}:=\bt:\bH:\bt,
     &                       &
    \kappa_{6}=I_{2}.
  \end{align*}
\end{thm}

\begin{rem}\label{rem:Etrigo}
  In this set, the invariant $\kappa_{6}=I_{2}=\tr \bd_{2}$ can be changed into $I_{3}=\tr \bd_{3}$ (by lemma~\ref{lem:Couple_tetra_tri}).
\end{rem}

\subsection{The special case of fourth-order harmonic tensors}

The theorems provided in~\autoref{sec:Ela} apply, of course, to fourth-order harmonic tensors $\bH\in \HH^{4}(\RR^{3})\subset \Ela$ (as the special case $\bd=\bv=0$). In the cubic case, the functional basis defined by the single invariant ${I_{3}}/{I_{2}}$ in proposition~\ref{prop:SepHcubic} is trivially recovered. In the transversely isotropic, tetragonal and trigonal cases, the functional bases provided in propositions~\ref{prop:SepHisoT},~\ref{prop:SepHtetra} and~\ref{prop:SepHtrigo} are recovered as special cases of theorems~\ref{thm:EIsoT},~\ref{thm:Etetra} and~\ref{thm:Etrigo}, thanks to the equalities $\bH:\bd_{2}=2 \bd_{2}^{\prime}$~\cite[Remark 5.2]{OKDD2018}, $\bd_{2}^{\prime}=\pm \sqrt{\frac{3}{2}} \norm{\bd_{2}^{\prime}} \bt\ne 0$ (see section \ref{subsec:ti-Ela-cov}) and the definitions \eqref{eq:Ikinvariants}, so that
\begin{equation}\label{eq:tHt-isoT-tetra-trigo}
  \bt:\bH:\bt =\frac{2}{3}\frac{\bd_{2}^{\prime}:\bH:\bd_{2}^{\prime}}{\norm{\bd_{2}^{\prime}}^{2}}=\frac{4}{3}\frac{\bd_{2}^{\prime}:\bd_{3}'}{I_{4}}=\frac{4I_{5}}{3I_{4}}.
\end{equation}
The equality $\bt:\bH:\bt =4 I_{5}/3I_{4}$ is valid for any transversely isotropic, tetragonal or trigonal pair $(\bH, \bt)$ with $\norm{\bt'}=\sqrt{{2}/{3}}$ .

\section{A polynomial functional basis for elasticity tensors at least tetragonal or trigonal}
\label{sec:AllFivesClasses}

By remark \ref{rem:non-sep-closed-stratum}, each functional basis obtained in the previous section is \emph{a priori} valid for one and only one elasticity symmetry stratum, among the cubic, the transversely isotropic, the tetragonal and the trigonal ones\footnote{The two invariants $\tr \bd, \tr \bv$ constitute a minimal functional basis for  the isotropic stratum $\strata{\SO(3)}$.}.
Consider now the union of strata
\begin{equation*}
  \SX:=\strata{\SO(3)} \cup \strata{\octa} \cup \strata{\OO(2)} \cup \strata{\DD_{3}} \cup \strata{\DD_{4}} \subset \Ela.
\end{equation*}
Of course, for each of these strata, the set of numerators and denominators of the rational invariants involved in their respective rational separating set obtained, constitutes a separating set of polynomial invariants for each of them. But the union of these sets is not separating for $\SX$ (see remark \ref{rem:non-sep-closed-stratum}). The question is thus whether one can merge and complete these separating sets in order to build a polynomial separating set, and hence a polynomial functional basis, valid for any elasticity tensors $\bE$ at least tetragonal or trigonal, \emph{i.e.}, for $\SX$. A positive response is provided by the following result (see~\autoref{sec:proof-main-theorem} for a proof).

\begin{thm}\label{thm:finalE}
  Let $\bE=(\tr \bd, \tr \bv, \bd^{\prime}, \bv^{\prime}, \bH)$ be an elasticity tensor and
  \begin{align*}
     & K_{1}:=\tr \bd,
     &                     & L_{1}:=\tr \bv,
     &                     & I_{3}:= \tr \bd_{3},
     &                     & K_{4}:=\norm{\bd^{\prime}}^{4} + \norm{\bv^{\prime}}^{4} + \norm{\bd_{2}^{\prime}}^{2},
    \\
     & K_{5}:=\bd:\bk_{4},
     &                     & L_{5}:= \bv:\bk_{4},
     &                     & K_{9}:=\bk_{4}:\bH:\bk_{4},
     &                     & K_{10}:=\norm{\tr(\bH\times \bk_{4})}^{2}.
  \end{align*}
  where $\bk_{4}:= (\norm{\bd^{\prime}}^{2}\, {\bd^{\prime}}^{2}+ \norm{\bv^{\prime}}^{2}\,{\bv^{\prime}}^{2}+{\bd_{2}^{\prime}}^{2})'$.
  \begin{enumerate}
    \item A minimal functional basis for $\strata{\SO(3)} \cup \strata{\octa} \cup \strata{\OO(2)} \cup \strata{\DD_{4}}$ (\emph{i.e.}, at least tetragonal elasticity tensors) consists of the seven polynomial invariants
          $K_1, L_1, I_3, K_4, K_{5}, L_{5}$ and $K_{9}$.

    \item A minimal functional basis for $\strata{\SO(3)} \cup \strata{\octa} \cup \strata{\OO(2)} \cup \strata{\DD_{3}}$ (\emph{i.e.}, at least trigonal elasticity tensors) consists of the seven polynomial invariants
          $K_1, L_1, I_3, K_4, K_{5}, L_{5}$ and $K_{9}$.

    \item A minimal functional basis for $\SX$ (\emph{i.e.}, at least tetragonal or trigonal elasticity tensors) consists of the eight polynomial invariants $K_1, L_1, I_3, K_4, K_{5}, L_{5}, K_{9}$ and $K_{10}$.
  \end{enumerate}
\end{thm}

\section{Conclusion}

We have summarized the mathematical material that allows to define the notion of \emph{minimal functional basis}, not only on a whole vector space (such as $\Ela$) but also \emph{on its symmetry strata} $\strata{H}$. Restricting the concept of functional basis to the class of \emph{continuous} functions, we have been able to define a lower bound for the cardinality of such a basis for a stratum $\strata{H}$ (namely, $\dim \strata{H}/G$, where $\strata{H}/G$ is the orbit space of the symmetry strata $\strata{H}$), and formulate a method to produce such a minimal functional basis of $\strata{H}$ (theorem~\ref{thm:Key_Thm_Separating_Set}). Using this tool, we have been able to produce low-cardinality  minimal functional bases for the tetragonal, trigonal, transversely isotropic, and cubic strata of $\Ela$, of cardinal at most 6 (whereas a known integrity basis of the full space $\Ela$ contains 297 invariants~\cite{OKA2017}). Finally, theorem~\ref{thm:finalE} provides a minimal polynomial functional basis for the elasticity tensors which are at least tetragonal or trigonal, and consists of eight invariants.

\appendix

\section{Elasticity symmetry groups}
\label{sec:Ela-symmetry-groups}

For each of the eight symmetry classes of the elasticity tensor, as detailed in~\cite{FV1996}, we provide an explicit representative subgroup $H\subset \SO(3)$ in this class, which serves as a prototype for visualising each of these symmetries.
\begin{itemize}
  \item $\triv$ is the subgroup of $\SO(3)$ reduced to the identity element;
  \item $\ZZ_{2}$ is generated by the second-order rotation $\rot(\ee_{3},\pi)$. It has order $2$;
  \item $\DD_{2}$ is generated by the second-order rotations $\rot(\ee_{3},\pi)$ and $\rot(\ee_{1},\pi)$. It has order $4$;
  \item $\DD_{3}$ is generated by the third order rotation $\rot(\ee_{3},\frac{2\pi}{3})$ and the second-order rotation $\rot(\ee_{1},\pi)$. It has order $6$;
  \item $\DD_{4}$ is generated by the fourth-order rotation $\rot(\ee_{3},\frac{\pi}{2})$ and the second-order rotation $\rot(\ee_{1},\pi)$. It has order $8$;
  \item $\octa$ is the \emph{octahedral} group, the orientation-preserving symmetry group of the cube with vertices $(\pm 1,\pm 1,\pm 1)$.
        Its principal directions are the normals to its faces, which are the basis vectors $\pm \ee_{i}$. It has order 24;
  \item $\OO(2)$ is the subgroup generated by all rotations $\rot(\ee_{3},\theta)$ ($\theta\in [0;2\pi[$) and the second-order rotation $\rot(\ee_{1},\pi)$. It is of infinite order.
\end{itemize}

All these subgroups are compact. The notation $\rot(\nn,\theta)$ denotes a rotation of angle $\theta$ around axis $\langle \nn \rangle$.

\section{Stratification of fourth-order harmonic tensors}
\label{sec:stratification-H4}

The vector space $\HH^{4}(\RR^{3})$, of fourth-order harmonic tensors, splits into the following eight symmetry classes (the same as for the elasticity tensor), resulting into the following isotropy stratification of $\HH^{4}(\RR^{3})$ \cite{IG1984}
\begin{equation*}
  \HH^{4}(\RR^{3}) = \strata{\triv} \cup \strata{\ZZ_{2}} \cup \strata{\DD_{2}} \cup \strata{\DD_{3}} \cup \strata{\DD_{4}} \cup \strata{\OO(2)} \cup \strata{\octa} \cup \strata{\SO(3)},
\end{equation*}
namely into triclinic, monoclinic, orthotropic, trigonal, tetragonal, transversely isotropic, cubic and isotropic strata.

Necessary and sufficient covariant conditions characterizing each symmetry stratum of $\HH^{4}(\RR^{3})$ have been derived in~\cite[Theorems 9.3, 9.11, 9.15, and Corollary 9.7]{OKDD2021}. Some of these conditions which are necessary for our purpose are stated below, as theorem~\ref{thm:H-stratification-criteria} (recall that $\bd_{2}$ is transversely isotropic, \emph{iff} $\bd_{2} \times {\bd_{2}}^{2}=0$ and $\bd_{2}^{\prime}\ne 0$).
The conditions for the orthotropic and monoclinic cases require the introduction of additional covariants,
\begin{equation*}
  \bc_{3}:=\bH:\bd_2=2 \bd_3',
  \qquad
  \bc_{4}:=\bH:\bc_3,
  \qquad
  \vv_5:=\lc:[\bd_2, \bc_3],
  \qquad
  \vv_6:=\lc:[\bd_2, \bc_4],
\end{equation*}
of order 2 (the $\bc_k$) and order 1 (the $\vv_k$).

\begin{thm}\label{thm:H-stratification-criteria}
  Let $\bH\in \HH^{4}(\RR^{3})$ be a fourth-order
  harmonic tensor. Then
  \begin{enumerate}
    \item $\bH$ is \emph{isotropic} \emph{iff} $\bH = 0$ (\textit{i.e.} $I_{3} = I_{4} = 0$);
    \item $\bH$ is \emph{cubic} \emph{iff} $\bH \ne 0$ and $\bd_{2}$ is isotropic ($\bd_{2}\ne 0$ and $\bd_{2}^{\prime}=0$);
    \item $\bH$ is \emph{transversely isotropic} \emph{iff} $\bd_{2}$ is transversely isotropic and
          $\bH \times \bd_{2} = 0$;
    \item $\bH$ is \emph{tetragonal} \emph{iff} $\bd_{2}$ is transversely isotropic,
          $ \bH\times \bd_{2} \ne 0$, and $\left.\tr(\bH\times \bd_{2}) = 0\right.$;
    \item $\bH$ is \emph{trigonal} \emph{iff} $\bd_{2}$ is transversely isotropic,
          $\tr(\bH \times \bd_{2}) \ne 0$,
          and
          $\left.(\bH : \bd_{2}) \times \bd_{2} = 0\right.$;
    \item $\bH$ is \emph{orthotropic} \emph{iff} $\vv_{5} = \vv_{6} = 0$ and the pair $(\bd_{2},\bc_{3})$ is orthotropic;
    \item $\bH$ is \emph{monoclinic} \emph{iff} the triplet $(\bd_{2}, \bc_{3}, \bc_{4})$ is monoclinic.
    \item $\bH$ is \emph{triclinic} \emph{iff} the triplet $(\bd_{2}, \bc_{3}, \bc_{4})$ is triclinic.
  \end{enumerate}
\end{thm}

\begin{rem}\label{rem:J2K4}
  Polynomial equations involving invariants instead of covariants have been formulated in~\cite{AKP2014}, for some symmetry strata of $\HH^{4}(\RR^{3})$ (those of dimension at most $3$). They consist in a finite set of polynomial relations and inequalities on the $I_{k}$. For instance, we have $\bH=0 \iff \bd_{2}=\bz \iff I_{2}=\norm{\bH}^{2}=0$, $\bd_{2}^{\prime}= 0 \iff I_{4}=\norm{\bd_{2}^{\prime}}^{2}= 0$
  and, by~\eqref{eq:TrHd2_et_Disc}, we get that $\tr (\bH\times \bd_{2})=0 \iff 2I_{2}^{3}-60I_{3}^{2}-9I_{4}I_{2}+18I_{6}=0$.
  The condition $I_{4}=0$ characterizes the symmetry classes which are at least cubic, and we have $I_{4}\ne 0$ for each fourth-order harmonic tensor which is either transversely isotropic, tetragonal or trigonal.
\end{rem}

\section{A reconstruction formula}
\label{sec:reconstruction-formula}

We propose a reconstruction formula for each transversely isotropic tensor $\bH\in \strata{\OO(2)}$ by means of a transversely isotropic second-order tensor $\bt$. Denoting by $\langle\nn\rangle$ (where $\norm{\nn}=1$), the axis of transverse isotropy, we introduce $\bt:=(\nn\otimes \nn)'$, which belongs to $\HH^{2}(\RR^{3})$. Using the concept of \emph{harmonic square} introduced in~\cite{OKDD2018}, which builds a fourth-order harmonic tensor $\bt \ast \bt$ from a second-order harmonic tensor $\bt$, we get
\begin{equation*}
  \bt \ast \bt:=\bH(\bt \odot \bt)= \bt \odot \bt - \frac{4}{7}\, \bq \odot \bt^{2} + \frac{2}{35} \norm{\bt}^{2}\, \bq \odot \bq\; \in \HH^{4}(\RR^{3}),
\end{equation*}
where $\odot$ is the symmetric tensor product and $\bH(\bS)$, defined by ~\eqref{eq:H4E}, is the projection of a totally symmetric fourth-order tensor $\bS$ onto its fourth-order harmonic component $\bH$. It is such that
\begin{equation}\label{eq:normett}
  \norm{\bt \ast \bt}^{2}=\bt:(\bt \ast \bt): \bt=\frac{8}{35},\qquad	(\bt \ast \bt)\3dots  (\bt \ast \bt) =\frac{8}{105}\,\bq +\frac{12}{147}\,\bt.
\end{equation}
We have then a reconstruction formula for $\bH$, using the scalar $\bt:\bH:\bt$ and the deviatoric transversely isotropic second-order tensor $\bt$:

\begin{thm}\label{thm:ReconstIT}
  Each fourth-order harmonic tensor $\bH\in \strata{\OO(2)}$, transversely isotropic of axis $\langle \nn \rangle$, with $\norm{\nn}=1$, can be written
  \begin{equation}\label{eq:ReconstIT}
    \bH =\frac{35}{8} (\bt:\bH:\bt)\,\bt \ast \bt , \qquad
    \bt:\bH:\bt =\frac{28}{9}\frac{I_{3}}{I_{2}}=\frac{4}{3}\frac{I_{5}}{I_{4}},
  \end{equation}
  with $\bt=(\nn\otimes\nn)'$, and
  \begin{equation}\label{eq:ReconstITd2}
    \bd_{2}(\bH)= \frac{5}{48}  (\bt:\bH:\bt)^{2}\left(14\, \bq+15\, \bt\right).
  \end{equation}
\end{thm}

\begin{proof}
  It has been shown in~\cite[Theorem 5.2]{OKDD2018} that every transversely isotropic fourth-order harmonic tensor $\bH\in \HH^{4}(\RR^{3})$ can be reconstructed as
  \begin{equation}\label{eq:HisoT_d2}
    \bH = \frac{63}{25 \, I_{3}}\, \bd_{2}^{\prime}\ast \bd_{2}^{\prime}, \qquad I_{3}=\tr \bd_{3}.
  \end{equation}
  The evaluations~\eqref{eq:JkIT} of the invariants $I_{k}$ for all $\bH\in\strata{\OO(2)}$ give
  ${63}/{25I_{3}}={35I_{5}}/{9I_{4}^{2}}$ and $28I_{3}/9I_{2}=4I_{5}/3I_{4}$. We have thus
  \begin{equation*}
    \bH=\frac{63}{25 \, I_{3}}\, \bd_{2}^{\prime}\ast \bd_{2}^{\prime}=
    \frac{35I_{5}}{9I_{4}^{2}}\left(\bd_{2}^{\prime}\ast \bd_{2}^{\prime}\right)=
    \frac{35}{8} (\bt:\bH:\bt)\,\bt \ast \bt,
  \end{equation*}
  where the last equality results from $\bt:\bH:\bt ={4I_{5}}/{3I_{4}}$ (see~\eqref{eq:tHt-isoT-tetra-trigo}). We get thus~\eqref{eq:ReconstIT}. Finally, since $\bd_{2}(\bH)=\bH \3dots \bH$, we deduce~\eqref{eq:ReconstITd2} from~\eqref{eq:ReconstIT} and~\eqref{eq:normett}.
\end{proof}

\section{Separating sets for a pair $(\bH,\bt)$}
\label{sec:factorisation-result}

We provide here separating sets for a pair $(\bH, \bt)$, on the union of strata
\begin{equation*}
  \strata{\OO(2)} \cup \strata{\DD_{3}} \cup \strata{\DD_{4}},
\end{equation*}
where $\bH$ is a fourth-order harmonic tensor and $\bt$ is a transversely isotropic deviator.

\begin{lem}\label{lem:Couple_tetra_tri}
  Let $\bH\in \HH^{4}(\RR^{3})$ be a fourth-order harmonic tensor and $\bt=\left(\nn\otimes \nn\right)'$
  with $\norm{\nn}=1$, a deviatoric transversely isotropic tensor. If the pair $(\bH,\bt)$ is at least tetragonal, then all the $I_{k}(\bH)$ defined by~\eqref{eq:Ikinvariants} are polynomial functions of $I_{2}(\bH)$ and $\bt:\bH:\bt$. In particular
  \begin{equation}\label{eq:sig2-tetra}
    I_{3} = \frac{3}{4}(\bt:\bH:\bt)I_{2} - \frac{15}{8}(\bt:\bH:\bt)^{3}, \quad I_{4} = \frac{1}{6}\left(I_{2} - \frac{15}{2}(\bt:\bH:\bt)^{2}\right)^{2}.
  \end{equation}
  The same result holds if the pair $(\bH,\bt)$ is at least trigonal, but with
  \begin{equation}\label{eq:sig2-trigo}
    I_{3} = -\frac{9}{8}(\bt:\bH:\bt)I_{2} + \frac{405}{64}(\bt:\bH:\bt)^{3}, \quad I_{4} = \frac{1}{96}\left(I_{2} -
    \frac{135}{8}
    (\bt:\bH:\bt)^{2}
    \right)^{2}.
  \end{equation}
\end{lem}

\begin{proof}
  Suppose first that $(\bH,\bt)$ is at least tetragonal. Without loss of generality, we can assume that its symmetry group contains $\DD_{4}$ (defined in~\autoref{sec:Ela-symmetry-groups}), and thus that $\bt=(\ee_{3}\otimes \ee_{3})'$. Using the Kelvin representation~\eqref{eq:tetragonal_case} of $\bH$, we get $\bt:\bH:\bt = H_{3333}=8\delta$
  and \eqref{eq:Jktetra},
  from which we deduce~\eqref{eq:sig2-tetra}. Besides, each invariant $I_{k}(\bH)$ can be expressed as a polynomial function of $\delta$ and $\sigma^{2}$ by~\eqref{eq:Jktetra}, and thus of $\bt:\bH:\bt$ and $I_{2}$, which concludes the proof for the tetragonal case. The proof for the trigonal case is similar, except that the Kelvin representation~\eqref{eq:trigonal_case}
  leads to $\bt:\bH:\bt = H_{3333}=8\delta$
  and \eqref{eq:Jktrigo},
  and thus to~\eqref{eq:sig2-trigo}.
\end{proof}

The following theorem is a corollary of lemma~\ref{lem:Couple_tetra_tri} and of result~\cite[Lemma 8.8]{OKDD2021}, which we recall now.

\begin{lem}\label{lem:tetragonal-pair}
  Let $\bt$ be a transversely isotropic second-order tensor and $\bH \in \HH^{4}(\RR^{3})$. Then, $(\bH,\bt)$ is at least tetragonal if and only if $\tr (\bH \times \bt) = 0$.
\end{lem}

\begin{thm}\label{thm:H-t}
  Let $\bH\in \HH^{4}(\RR^{3})$ be a fourth-order harmonic tensor and $\bt=\left(\nn\otimes \nn\right)'$
  with $\norm{\nn}=1$, a deviatoric transversely isotropic tensor. Then, the set of invariants
  \begin{equation}\label{eq:sep-H-t}
    I_{3} := \tr \bd_{3}, \quad I_{4} := \tr {\bd_{2}^{\prime}}^{2} = \norm{\bd_{2}^{\prime}}^{2}, \quad  \norm{\tr(\bH\times \bt)}^{2}, \quad \bt:\bH:\bt
  \end{equation}
  is separating for the pair $(\bH,\bt)$ on $\strata{\OO(2)} \cup \strata{\DD_{3}} \cup \strata{\DD_{4}}$.
\end{thm}

\begin{proof}
  Let $(\bH,\bt)$ and $(\overline{\bH},\overline{\bt})$ in $\strata{\OO(2)} \cup \strata{\DD_{3}} \cup \strata{\DD_{4}}$, with $\bt = \left(\nn\otimes \nn\right)'$ and $\overline{\bt} = \left(\overline{\nn}\otimes \overline{\nn}\right)'$, and assume that the four invariants~\eqref{eq:sep-H-t} have the same values on $(\bH,\bt)$ and $(\overline{\bH},\overline{\bt})$. Suppose first that $\tr(\overline{\bH}\times \overline{\bt}) = \tr(\bH\times \bt) = 0$, then, by lemma~\ref{lem:tetragonal-pair}, we conclude that both $(\bH,\bt)$ and $(\overline{\bH},\overline{\bt})$ are at least tetragonal. Then, there exists a rotation $g$ such that $g\star \nn = \ee_{3}$ and $[g \star \bH]$ can be written as~\eqref{eq:tetragonal_case} with parameters $(\delta,\sigma)$. Similarly, there exists a rotation $\overline{g}$ such that $\overline{g}\star \overline{\nn} = \ee_{3}$ and $[\overline{g} \star \overline{\bH}]$ can be written as~\eqref{eq:tetragonal_case} with parameters $(\overline{\delta},\overline{\sigma})$. Now, by lemma~\ref{lem:Couple_tetra_tri}, we deduce that, in any case we have $(\overline{\delta},\overline{\sigma}) = (\delta,\pm\sigma)$ and
  thus that either
  $(\overline{\bH},\overline{\bt}) = (\overline{g}^{-1}g) \star (\bH,\bt)$ or $(\overline{\bH},\overline{\bt})
    = (\overline{g}^{-1}r g) \star (\bH,\bt)$, where $r= \rot(\ee_{3}, \frac{\pi}{4})$ is the rotation of angle $\pi/4$ about  $\ee_{3}$.
  If $\tr(\overline{\bH}\times \overline{\bt}) = \tr(\bH\times \bt) \ne 0$, then, $(\bH,\bt)$ and $(\overline{\bH},\overline{\bt})$ are both at least trigonal, and the arguments are similar.
\end{proof}

\section{Proof of theorem~\ref{thm:finalE}}
\label{sec:proof-main-theorem}

Observe first that an elasticity tensor $\bE\in\SX$ is at least cubic if and only if
\begin{equation*}
  K_{4} := \norm{\bd^{\prime}}^{4} + \norm{\bv^{\prime}}^{4} + \norm{\bd_{2}^{\prime}}^{2} = 0.
\end{equation*}
If $K_{4} \ne 0$, then, $\bE= (\tr \bd,\tr \bv,\bd^{\prime},\bv^{\prime},\bH)$ belongs thus to $\strata{\DD_{3}} \cup \strata{\DD_{4}} \cup \strata{\OO(2)}$, the transversely isotropic rational covariant $\bt := 2\bk_{4}/K_{4}$ is well defined and by~\eqref{eq:dvd2t}, and we have
\begin{equation}\label{eq:TensOrd2_Fonc_KL}
  \bd^{\prime}=\frac{3K_{5}}{K_{4}}\bt, \qquad \bv^{\prime}=\frac{3L_{5}}{K_{4}}\bt, \qquad \bd_{2}^{\prime}=\pm \sqrt{\frac{3}{2}}\norm{\bd_{2}^{\prime}}\; \bt.
\end{equation}
Moreover, $\bE$ has the same symmetry class as the pair $(\bH,\bt)$ which is either trigonal, tetragonal or transversely isotropic.

By theorems~\ref{thm:ECubic} to \ref{thm:Etrigo}, the set $\mathscr{F} := \set{K_{1},L_{1},\dotsc,K_{10}}$ is separating for each individual stratum $\strata{H}$ contained in $\SX$. Therefore, given two elasticity tensors $\bE,\overline{\bE} \in \SX$ with the same eight invariants $K_{1}= \overline K_{1}, L_{1}=\overline L_{1}, \dots, K_{10}=\overline K_{10}$, to prove that they are in the same orbit it is enough to show that they belong to the same symmetry class. Therefore, let $\overline{\bE} = (\tr \overline{\bd},\tr \overline{\bv},\overline{\bd^{\prime}},\overline{\bv^{\prime}},\overline{\bH})$, and we will argue according to the symmetry class of $\bE$.

\begin{enumerate}[(A)]
  \item If $\bE$ is isotropic, then, all invariants in $\mathscr{F}\setminus\set{K_{1},L_{1}}$ vanish. Hence,
        \begin{equation*}
          K_{4} = \overline{K}_{4} = \norm{\overline{\bd^{\prime}}}^{4} + \norm{\overline{\bv^{\prime}}}^{4} + \norm{\overline{\bd_{2}^{\prime}}}^{2} = 0 \implies
          \overline{\bd^{\prime}} = \overline{\bv^{\prime}} = \overline{\bd_{2}^{\prime}} = \mathbf{0}.
        \end{equation*}
        Therefore, by theorem~\ref{thm:H-stratification-criteria}, $\overline{\bH}$ is at least cubic, and since $I_{3}=\overline{I}_{3}=0$, we conclude by~\eqref{eq:Jkcubic}, that $\overline{\bH}= 0$, and thus that $\overline{\bE}$ is isotropic.
  \item If $\bE$ is cubic, then, all invariants in $\mathscr{F}\setminus\set{K_{1},L_{1},I_{3}}$ vanish but $I_{3}=\overline I_{3}\ne 0$. We conclude, as in case (A), that $\overline{\bE}$ is at least cubic, and indeed cubic, since $\overline I_{3}\ne 0$ and thus $\overline I_{2} = \norm{\overline{\bH}}^{2} \ne 0$.
  \item If $\bE$ is either transversely isotropic, trigonal or tetragonal, then, $\overline{K_{4}} = K_{4} \ne 0$
        and thus $\overline{\bE}$ is either transversely isotropic, trigonal or tetragonal. Hence, $\overline{\bt}$ is well-defined and is written
        \begin{equation*}
          \overline{\bt} := \frac{2}{\overline{K}_{4}}(\norm{\overline{\bd^{\prime}}}^{2}\, {\overline{\bd^{\prime}}}^{2}+ \norm{\overline{\bv^{\prime}}}^{2}\,{\overline{\bv^{\prime}}}^{2}+{\overline{\bd_{2}^{\prime}}}^{2})^{\prime},
        \end{equation*}
        and $\overline{\bE}$ has the same symmetry class as the pair $(\overline{\bH},\overline{\bt})$.
        Now, from~\eqref{eq:TensOrd2_Fonc_KL} and since $\norm{\bt}^2=\norm{\overline{\bt}}^2=2/3$, $K_{5}=\overline{K}_{5}$, $L_{5}=\overline{L}_{5}$,
        \begin{equation*}
          K_{4} = \frac{36}{K_{4}^4}\left(K_{5}^4 + L_{5}^4 \right) + \norm{\bd_{2}^{\prime}}^{2}
          = \overline{K}_{4} = \frac{36}{\overline{K}_{4}^4}\left(\overline{K}_{5}^4 + \overline{L}_{5}^4 \right) + \norm{\overline{\bd_{2}^{\prime}}}^{2},
        \end{equation*}
        we get $\overline{I}_{4}=\norm{\overline{\bd_{2}^{\prime}}}^{2} = \norm{\bd_{2}^{\prime}}^{2}=I_{4}$.
        Therefore, by theorem~\ref{thm:H-t}, $(\bH, \bt)$ and $(\overline{\bH}, \overline{\bt})$
        are in the same orbit, and belong thus to the same symmetry class.
\end{enumerate}
So far, we have proved that the family $\mathscr{F}=\set{K_{1},L_{1},\dotsc,K_{10}}$ is separating for the three cases (1), (2) and (3) of theorem~\ref{thm:finalE}. Note moreover, that $K_{10}$ can be removed from $\mathscr{F}$ when the goal is to separate  "at least tetragonal" (case (1)) or "at least trigonal" (case (2)) elasticity tensors. Indeed,  in the proof of theorem~\ref{thm:H-t}, $K_{10}$ is used only to distinguish whether a pair $(\bH, \bt)$ is tetragonal or trigonal.

Finally, it remains to prove the minimality of the separating set $\mathscr{F}$ for $\SX$, and of $\mathscr{F}\setminus \set{K_{10}}$ for cases (1) and (2). Let $\mathscr{F}'$ be a proper subset of $\mathscr{F}$ or of $\mathscr{F}\setminus \set{K_{10}}$.
\begin{enumerate}[(a)]
  \item If $\mathscr{F}'$ does not contain $K_{1}=\tr \bd$ or $L_{1}=\tr \bv$, then it fails to be a separating set for isotropic elasticity tensors.
  \item If $\mathscr{F}'$ does not contain $I_{3}=\tr \bd_{3}$, then it fails to be a separating set for cubic elasticity tensors.
  \item If $\mathscr{F}'$ does not contain $K_{5}=\bd:\bk_{4}$, then the two transversely isotropic elasticity tensors $\bE_{1}=(0,0, \bd^{\prime}, 0, 0)$ and
        $\bE_{2}=(0,0, -\bd^{\prime}, 0, 0)$ have the same values for $\mathscr{F}'$ but are not on the same orbit. The same conclusion holds for $L_{5}= \bv:\bk_{4}$.
  \item If $\mathscr{F}'$ does not contain $K_{4}=\norm{\bd^{\prime}}^{4} +\norm{\bv^{\prime}}^{4}+I_{4}$, then  it fails to separate an harmonic tetragonal tensor $\bH_{\DD_{4}}$ with $\delta=0$ (see~\eqref{eq:tetragonal_case}) from $\bE=0$,
        since all the invariants in $\mathscr{F}'$ vanish on these tensors.
  \item If $\mathscr{F}'$ does not contain $K_{9}=\bk_{4}:\bH:\bk_{4}$, it fails to separate tetragonal harmonic tensors
        (by remark~\ref{I2I3_nor_I3I4_tetra}) and to separate trigonal harmonic tensors
        (by remark~\ref{I2I3_nor_I3I4_trigo}).
  \item Finally, if $\mathscr{F}'$ does not contain $K_{10}=\norm{\tr(\bH\times \bk_{4})}^{2}$, then it fails to be a separating set for $\SX$ since all the invariants in
        $\mathscr{F}'$ take the same values
        on the harmonic trigonal tensor $\bH_{\DD_{3}}$
        with $\delta=0$ and $\sigma=\sigma_{1}\ne 0$ (see~\eqref{eq:trigonal_case})
        and the harmonic tetragonal tensor $\bH_{\DD_{4}}$ with $\delta=0$ and $\sigma=\sigma_{2}$
        (see~\eqref{eq:tetragonal_case}), when $\sigma_{2}^{2}=2 \sigma_{1}^{2}$.
\end{enumerate}
These arguments show that $\mathscr{F}$ is a minimal separating set for $\SX$, which proves point (3) of theorem~\ref{thm:finalE}. Items (a) to (e) show that $\mathscr{F}\setminus \set{K_{10}}$ is a minimal separating set for either at least tetragonal tensors (point (1)) or
at least trigonal tensors (point (2)), which achieves the proof.

\subsection*{Funding} The authors were partially supported by CNRS Projet 80--Prime GAMM (Géométrie algébrique complexe/réelle et mécanique des matériaux).



\begin{thebibliography}{10}

\bibitem{ADDKO2019}
S.~Abramian, B.~Desmorat, R.~Desmorat, B.~Kolev, and M.~Olive.
\newblock Recovering the normal form and symmetry class of an elasticity
  tensor.
\newblock {\em J. Elasticity}, 142(1):1--33, jul 2020.

\bibitem{AS1983}
M.~Abud and G.~Sartori.
\newblock {T}he geometry of spontaneous symmetry breaking.
\newblock {\em Ann. Physics}, 150(2):307--372, 1983.

\bibitem{AKP2014}
N.~Auffray, B.~Kolev, and M.~Petitot.
\newblock {O}n {A}nisotropic {P}olynomial {R}elations for the {E}lasticity
  {T}ensor.
\newblock {\em J. Elasticity}, 115(1):77--103, June 2014.

\bibitem{Bac1970}
G.~Backus.
\newblock {A} geometrical picture of anisotropic elastic tensors.
\newblock {\em Rev. Geophys.}, 8(3):633--671, 1970.

\bibitem{Bae1993}
R.~Baerheim.
\newblock {H}armonic decomposition of the anisotropic elasticity tensor.
\newblock {\em Q. J. Mech. Appl. Math.}, 46(3):391--418, 1993.

\bibitem{Boe1987}
J.-P. Boehler.
\newblock Introduction to the invariant formulation of anisotropic constitutive
  equations.
\newblock In {\em Applications of tensor functions in solid mechanics}, volume
  292 of {\em CISM Courses and Lectures}, pages 13--30. Springer, Vienna, 1987.

\bibitem{BKO1994}
J.-P. Boehler, A.~A. Kirillov, Jr., and E.~T. Onat.
\newblock {O}n the polynomial invariants of the elasticity tensor.
\newblock {\em J. Elasticity}, 34(2):97--110, 1994.

\bibitem{BBS2008}
A.~Bóna, I.~Bucataru, and M.~A. Slawinski.
\newblock {S}pace of ${SO}(3)$-orbits of elasticity tensors.
\newblock {\em Arch. Mech. (Arch. Mech. Stos.)}, 60(2):123--138, 2008.

\bibitem{Bredon1960}
G.~E. Bredon.
\newblock Finiteness of number of orbit types.
\newblock In {\em A. Borel, Seminar on transformation groups. With
  contributions by G. Bredon, EE Floyd, D. Montgomery, R. Palais. Annals of
  Mathematics Studies}, number~46, 1960.

\bibitem{Bre1972}
G.~E. Bredon.
\newblock {\em {I}ntroduction to compact transformation groups}.
\newblock Academic Press, New York, 1972.
\newblock Pure and Applied Mathematics, Vol. 46.

\bibitem{Bredon1993}
G.~E. Bredon.
\newblock {\em Topology and geometry}.
\newblock Graduate Texts in Mathematics. Springer, corrected edition, 1993.

\bibitem{Brouwer1912}
L.~Brouwer.
\newblock Beweis der invarianz des n-dimensionalen gebiets.
\newblock {\em Math. Ann.}, 71:305--313, 1912.

\bibitem{CLQZZ2018}
Z.~Chen, J.~Liu, L.~Qi, Q.~Zheng, and W.~Zou.
\newblock An irreducible function basis of isotropic invariants of a third
  order three-dimensional symmetric tensor.
\newblock {\em J. Math. Phys.}, 59:081703, 2018.

\bibitem{Cow1989a}
S.~C. Cowin.
\newblock {P}roperties of the anisotropic elasticity tensor.
\newblock {\em Q. J. Mech. Appl. Math.}, 42:249--266, 1989.

\bibitem{DADKO2019}
R.~Desmorat, N.~Auffray, B.~Desmorat, B.~Kolev, and M.~Olive.
\newblock Generic separating sets for three-dimensional elasticity tensors.
\newblock {\em Proc. R. Soc. A}, 475, 2019.

\bibitem{DKW2008}
J.~Draisma, G.~Kemper, and D.~Wehlau.
\newblock {P}olarization of separating invariants.
\newblock {\em Canad. J. Math.}, 60(3):556--571, 2008.

\bibitem{FV1996}
S.~Forte and M.~Vianello.
\newblock {S}ymmetry classes for elasticity tensors.
\newblock {\em J. Elasticity}, 43(2):81--108, 1996.

\bibitem{Hil1993}
D.~Hilbert.
\newblock {\em {T}heory of algebraic invariants}.
\newblock Cambridge University Press, Cambridge, 1993.

\bibitem{HY1988}
J.~G. Hocking and G.~S. Young.
\newblock {\em {T}opology}.
\newblock Dover Publications Inc., New York, 1988.

\bibitem{IG1984}
E.~Ihrig and M.~Golubitsky.
\newblock {P}attern selection with {${\rm O}(3)$} symmetry.
\newblock 13(1-2):1--33, 1984.

\bibitem{KP2000}
H.~Kraft and C.~Procesi.
\newblock {C}lassical {I}nvariant {T}heory, a {P}rimer.
\newblock Lectures notes avaiable at
  \url{http://www.math.unibas.ch/~kraft/Papers/KP-Primer.pdf}, 2000.

\bibitem{LDQZ2018}
J.~J. Liu, W.~Y. Ding, L.~Q. Qi, and W.~N. Zou.
\newblock Isotropic polynomial invariants of the hall tensor.
\newblock {\em Appl. Math. Mech.-Engl. Ed.}, 39(12):1845--1856, 2018.

\bibitem{MCQ+19}
Z.~Ming, Y.~Chen, L.~Qi, and L.~Zhang.
\newblock A polynomially irreducible functional basis of elasticity tensors.
\newblock {\em arXiv preprint arXiv:1912.03077}, 2019.

\bibitem{MZC2019}
Z.~Ming, L.~Zhang, and Y.~Chen.
\newblock An irreducible polynomial functional basis of two-dimensional eshelby
  tensors.
\newblock {\em Appl. Math. Mech.-Engl. Ed.}, 40(8):1169--1180, 2019.

\bibitem{Olive2017}
M.~Olive.
\newblock About {G}ordan's algorithm for binary forms.
\newblock {\em Found. Comput. Math.}, 17(6):1407--1466, 2017.

\bibitem{OA2014}
M.~Olive and N.~Auffray.
\newblock {I}sotropic invariants of a completely symmetric third-order tensor.
\newblock {\em Journal of Mathematical Physics}, 55(9):092901, 2014.

\bibitem{OKA2017}
M.~Olive, B.~Kolev, and N.~Auffray.
\newblock A minimal integrity basis for the elasticity tensor.
\newblock {\em Arch. Ration. Mech. Anal.}, 226(1):1--31, Oct. 2017.

\bibitem{OKDD2018}
M.~Olive, B.~Kolev, R.~Desmorat, and B.~Desmorat.
\newblock Harmonic factorization and reconstruction of the elasticity tensor.
\newblock {\em Journal of Elasticity}, 101:132--67, 2018.

\bibitem{OKDD2021}
M.~Olive, B.~Kolev, R.~Desmorat, and B.~Desmorat.
\newblock Characterization of the symmetry class of an elasticity tensor using
  polynomial covariants.
\newblock {\em Mathematics and Mechanics of Solids}, may 2021.

\bibitem{Olver1995}
P.~J. Olver.
\newblock {\em {E}quivalence, invariants, and symmetry}, volume~44.
\newblock Cambridge University Press, Cambridge, 1995.

\bibitem{PR1959}
A.~Pipkin and R.~S. Rivlin.
\newblock The formulation of constitutive equations in continuum physics {I}.
\newblock {\em Arch. Rational Mech. Anal.}, 4:129--144 (1959), 1959.

\bibitem{PW1963}
A.~Pipkin and A.~Wineman.
\newblock Material symmetry restrictions on non-polynomial constitutive
  equations.
\newblock {\em Arch. Rational Mech. Anal.}, 12:420--426, 1963.

\bibitem{Ryc1984}
J.~Rychlewski.
\newblock On hooke's law.
\newblock {\em Prikl. Matem. Mekhan.}, 48:303--314, 1984.

\bibitem{Shi1967}
T.~Shioda.
\newblock {O}n the graded ring of invariants of binary octavics.
\newblock {\em Amer. J. Math.}, 89:1022--1046, 1967.

\bibitem{Smi1971}
G.~Smith.
\newblock {O}n isotropic functions of symmetric tensors, skew-symmetric tensors
  and vectors.
\newblock {\em Int. J. Eng. Sci.}, 9:899--916, 1971.

\bibitem{Stu2008}
B.~Sturmfels.
\newblock {\em Algorithms in Invariant Theory}.
\newblock Texts \& Monographs in Symbolic Computation. 2\textsuperscript{nd}
  edition, Springer Wien New-York, 2008.

\bibitem{TKel1878}
W.~K. Thomson (Lord~Kelvin).
\newblock {\em Elasticity, Encyclopaedia Britannica}.
\newblock Adam and Charles Black, Edinburgh, 1878.

\bibitem{VV2001}
P.~Vannucci and G.~Verchery.
\newblock Stiffness design of laminates using the polar method.
\newblock {\em International Journal of Solids and Structures}, 38:9281--9894,
  2001.

\bibitem{Ver1979}
G.~Verchery.
\newblock Les invariants des tenseurs d'ordre 4 du type de
  l'{é}lasticit{é}.
\newblock In J.-P. Boehler, editor, {\em Colloque Int. CNRS 295, Villard de
  Lans}, pages 93--104. Martinus Nijhoff Publishers and Editions du CNRS, 1982,
  1979.

\bibitem{Via1997}
M.~Vianello.
\newblock {A}n integrity basis for plane elasticity tensors.
\newblock {\em Arch. Mech.}, 49:197--208, 1997.

\bibitem{Wan1970a}
C.-C. Wang.
\newblock {C}orrigendum to my recent papers on {R}epresentations for isotropic
  functions.
\newblock {\em Arch. Rational Mech. Anal.}, 43:392--395, 1970.

\bibitem{Wey1997}
H.~Weyl.
\newblock {\em {T}he classical groups}.
\newblock Princeton Landmarks in Mathematics. Princeton University Press,
  Princeton, NJ, 1997.
\newblock Their invariants and representations, Fifteenth printing, Princeton
  Paperbacks.

\bibitem{WP1964}
A.~Wineman and A.~Pipkin.
\newblock Material symmetry restrictions on constitutive equations.
\newblock {\em Arch. Ration. Mech. Anal.}, 17:184--214, 1964.

\bibitem{Zhe1994}
Q.-S. Zheng.
\newblock {T}heory of representations for tensor functions - {A} unified
  invariant approach to constitutive equations.
\newblock {\em Appl. Mech. Rev.}, 47:545--587, 1994.

\end{thebibliography}
\end{document}